\documentclass[10pt, a4paper, oneside, leqno]{article}

\usepackage[utf8]{inputenc}

\usepackage{amsthm}
\usepackage{amsmath}
\usepackage{amssymb}

\usepackage[square,comma]{natbib}

\usepackage[%
    bookmarks,
    colorlinks,
    breaklinks,
    linkcolor=blue,
    citecolor=blue,
    filecolor=blue,
    urlcolor=blue
]{hyperref}

\usepackage{doi}

\usepackage{braket}

\newbox\gnBoxA
\newdimen\gnCornerHgt
\setbox\gnBoxA=\hbox{$\ulcorner$}
\global\gnCornerHgt=\ht\gnBoxA
\newdimen\gnArgHgt
\def\gcode #1{%
\setbox\gnBoxA=\hbox{$#1$}%
\gnArgHgt=\ht\gnBoxA%
\ifnum     \gnArgHgt<\gnCornerHgt \gnArgHgt=0pt%
\else \advance \gnArgHgt by -\gnCornerHgt%
\fi \raise\gnArgHgt\hbox{$\ulcorner$} \box\gnBoxA %
\raise\gnArgHgt\hbox{$\urcorner$}}

\newtheorem{thm}{Theorem}[section]
\theoremstyle{definition}
\newtheorem{dfn}[thm]{Definition}
\theoremstyle{plain}
\newtheorem{lem}[thm]{Lemma}
\theoremstyle{plain}
\newtheorem{fact}[thm]{Fact}
\theoremstyle{plain}
\newtheorem{cor}[thm]{Corollary}

\newcommand{\Tleq}{\leq_{\mathbf{T}}}

\newcommand{\Nat}[0]{\mathbb{N}}

\newcommand{\PowN}[0]{\mathcal{P}(\mathbb{N})}

\newcommand{\lt}[0]{\mathrm{L}_2}
\newcommand{\z}[0]{\mathsf{Z}}

\newcommand{\rca}[0]{\mathsf{RCA}}

\newcommand{\wkl}[0]{\mathsf{WKL}}

\newcommand{\aca}[0]{\mathsf{ACA}}
\newcommand{\atr}[0]{\mathsf{ATR}}

\newcommand{\pica}[0]{\Pi^1_1\text{-}\mathsf{CA}}
\newcommand{\pra}[0]{\mathrm{PRA}}
\newcommand{\isig}[1]{\mathrm{I}\Sigma_{#1}}

\newcommand{\zfc}[0]{\mathrm{ZFC}}

\newcommand{\pa}[0]{\mathrm{PA}}

\newcommand{\arith}[0]{\mathrm{ARITH}}
\newcommand{\hyp}[0]{\mathrm{HYP}}

\DeclareMathOperator{\Con}{Con}

\title{Computational reverse mathematics \\ and foundational analysis}
\author{Benedict Eastaugh\footnote{%
Munich Center for Mathematical Philosophy,
LMU Munich,
Geschwister-Scholl-Platz 1,
80539 Munich,
Germany.
Email: \href{mailto:benedict@eastaugh.net}{benedict@eastaugh.net}
}}
\date{June 28, 2018}

\begin{document}

\maketitle

\begin{abstract}
    Reverse mathematics studies which subsystems of second order arithmetic are
    equivalent to key theorems of ordinary, non-set-theoretic mathematics. The
    main philosophical application of reverse mathematics proposed thus far is
    \emph{foundational analysis}, which explores the limits of different
    foundations for mathematics in a formally precise manner. This paper gives a
    detailed account of the motivations and methodology of foundational
    analysis, which have heretofore been largely left implicit in the practice.
    It then shows how this account can be fruitfully applied in the evaluation
    of major foundational approaches by a careful examination of two case
    studies: a partial realization of Hilbert's program due to
    \citet{Simpson1988}, and predicativism in the extended form due to Feferman
    and Schütte.

    \citet{Shore2010, Shore2013} proposes that equivalences in reverse
    mathematics be proved in the same way as inequivalences, namely by
    considering only $\omega$-models of the systems in question. Shore refers to
    this approach as \emph{computational reverse mathematics}. This paper shows
    that despite some attractive features, computational reverse mathematics is
    inappropriate for foundational analysis, for two major reasons. Firstly, the
    computable entailment relation employed in computational reverse mathematics
    does not preserve justification for the foundational programs above.
    Secondly, computable entailment is a $\Pi^1_1$ complete relation, and hence
    employing it commits one to theoretical resources which outstrip those
    available within any foundational approach that is proof-theoretically
    weaker than $\pica_0$.
\end{abstract}



\section{Introduction}
\label{sec:introduction}

In ordinary mathematical practice, mathematicians prove theorems, reasoning from
a fixed set of axioms to a logically derivable conclusion. The axioms in play
are usually implicit: mathematicians rarely assert at the beginning of their
papers that they work in, for example, $\pa$ or $\zfc$. Given a particular proof
we might ask which axioms were employed and thus make explicit the author's
assumptions. Now that we have a set of axioms $\Gamma$ which are sufficient to
prove some theorem $\varphi$, we could further inquire whether they are
necessary to prove the theorem, or whether a strictly weaker set of axioms would
suffice. To a first approximation, \emph{reverse mathematics} is the program of
discovering precisely which axioms are both necessary and sufficient to prove
any given theorem of ordinary mathematics.

Reverse mathematics was initiated by Harvey Friedman \citeyearpar{Friedman1975,
Friedman1976}, and extensively developed in the work of Stephen Simpson and his
students. It determines the proof-theor\-etic strength of theorems of ordinary
mathematics, by proving equivalences between formalised versions of those
theorems and axiom systems in a hierarchy of known strength. Roughly speaking,
the term ``ordinary mathematics'' means non-set-theor\-etic mathematics, i.e.
those parts of mathematics which do not depend on abstract set-theor\-etical
concepts. Typical examples of ordinary mathematics include real and complex
analysis, countable algebra, and the topology of complete separable metric
spaces.

The axiom systems used in reverse mathematics are subsystems of second order
arithmetic or $\z_2$. This is an extension of familiar first order systems of
arithmetic such as Peano arithmetic. In the intended interpretation, variables
in first order arithmetic range over the natural numbers $\mathbb{N}$. Second
order arithmetic also has number variables ranging over the natural numbers, but
in addition to these it has set variables which range over sets of numbers
$X \subseteq \mathbb{N}$. For the full technical background on second order
arithmetic the reader should consult \citet{Simpson2009}, the primary reference
work on reverse mathematics. Here we restrict ourselves to sketching the basic
features of the framework and explaining some salient details.

The language of second order arithmetic, $\lt$, is a two-sorted first order
language with the following nonlogical symbols: constant symbols $0$ and $1$,
binary function symbols $+$ and $\cdot$, and the binary relation symbols $<$ and
$\in$.
$\lt$-structures have two domains: a first order domain $|M|$ over which the
number variables $x_0, x_1, \dotsc$  range, and a second order domain
$\mathcal{S} \subseteq \mathcal{P}(|M|)$ over which the set variables $X_0, X_1,
\dotsc$ range. An $\lt$-structure $M$ is thus a tuple of the form
\begin{equation}
    \label{eqn:l2_structure}
    M = \langle
        |M|, \mathcal{S}_M,
        +_M, \cdot_M,
        0_M, 1_M,
        <_M
    \rangle,
\end{equation}
where $0_M$ and $1_M$ are elements of $|M|$, $+_M$ and $\cdot_M$ are functions
from $|M| \times |M|$ to $|M|$, and $<_M$ is a binary relation on $|M|$.

The formal system $\z_2$ of second order arithmetic has a long history in work
on foundations of mathematics, which we can trace back to Dedekind. The most
substantive classical developments are those of \citet{HilBer1968}. The axioms
of $\z_2$ fall into three groups: the basic axioms; a comprehension scheme; and
an induction axiom. The basic axioms are those of Peano arithmetic, minus the
induction scheme. To these is added the comprehension scheme
\begin{equation*}
    \label{eqn:z2_comp_scheme}
    \exists{X}\forall{n} ( n \in X \leftrightarrow \varphi(n) )
    \tag{$\mathsf{CA}$}
\end{equation*}
for all $\lt$-formulae $\varphi$ (with parameters). Many subsystems of second
order arithmetic are obtained by restricting this comprehension scheme to
particular syntactically-defined subclasses. Finally there is the induction
axiom
\begin{equation*}
    \label{eqn:z2_ind_axiom}
    \forall{X} (
        (0 \in X \wedge \forall{n} (n \in X \rightarrow n + 1 \in X))
        \rightarrow
        \forall{n} (n \in X)
    ).
    \tag{$\mathsf{I}_0$}
\end{equation*}
This is a single axiom rather than an axiom scheme. As such its strength is tied
directly to the associated comprehension scheme: we have induction only for
those sets which we can prove to exist by comprehension. Because $\z_2$ includes
the comprehension scheme for all $\lt$-formulae $\varphi$, every instance of the
second order induction scheme
\begin{equation*}
    \label{eqn:z2_ind_scheme}
    (\varphi(0) \wedge \forall{n} (\varphi(n) \rightarrow \varphi(n + 1)))
    \rightarrow
    \forall{n} \; \varphi(n)
    \tag{$\mathsf{I}$}
\end{equation*}
is a theorem of $\z_2$. By restricting $\varphi$ to formulae in the language of
first-order arithmetic $\mathrm{L}_1$ we obtain the induction scheme of
first-order Peano arithmetic ($\pa$). A stronger restriction, limiting $\varphi$
to $\Sigma^0_1$ formulae, gives us the $\Sigma^0_1$ induction scheme.

A subsystem $T$ of $\z_2$ is a formal system in the language $\mathrm{L}_2$ such
that each axiom $\varphi$ of $T$ is a theorem of $\z_2$. The central subsystems
are colloquially known as the \emph{Big Five}. Each of them includes the basic
axioms, the induction axiom and some other set existence axioms.
The weakest is $\rca_0$, the usual base theory for reverse mathematics. Its
axioms consist of the basic axioms, $\Delta^0_1$ (that is, recursive)
comprehension and $\Sigma^0_1$ induction. Induction in $\rca_0$ thus outstrips
comprehension, although it is still quite limited compared to $\pa$ or $\z_2$.
The other members of the Big Five are:
$\wkl_0$, obtained by adding to $\rca_0$ the assertion that every infinite
subtree of $2^{<\mathbb{N}}$ has an infinite path through it;
$\aca_0$, which is given by the comprehension scheme for all arithmetically
definable sets;
$\atr_0$, which extends $\aca_0$ with an axiom allowing the iteration of the
arithmetical operations along any wellordering;
and finally $\pica_0$, whose defining axiom is the $\Pi^1_1$ comprehension
scheme.

At first glance, second order arithmetic may seem somewhat limited, and
unsuitable for the development of large portions of mathematics, even when we
restrict our attention to ordinary, non-set-theoretic mathematics.
Formally, second order arithmetic includes only two kinds of entities: natural
numbers and sets of natural numbers. While clever coding schemes allow objects
from many branches of mathematics to be represented within this formally austere
framework, limitations abound, generally to do with cardinality: one cannot
quantify over uncountable sets of real numbers, prove theorems about topological
spaces of arbitrary cardinality, and so on. Most obviously, the set $\mathbb{R}$
of all real numbers cannot be directly represented.

Mathematics within second order arithmetic is thus limited to countable or
countably representable structures such as complete separable metric spaces and
countable abelian groups. Nevertheless, this turns out to include a wide variety
of mathematical objects including real numbers, continuous functions on the real
line and complex plane, and Borel and analytic sets. Constructing
representations of these objects and proving that they are well-behaved usually
requires a certain minimum of theoretical strength. Definitions in reverse
mathematics are therefore often given relative to a particular system, usually
the weak base system $\rca_0$.

While we can do enough in $\rca_0$ to get mathematics off the ground, many key
theorems require stronger axioms. A typical example is that $\rca_0$ does not
prove the Bolzano--Weierstraß theorem, a fundamental theorem in analysis which
states that every bounded sequence of real numbers has a convergent subsequence.
The Bolzano--Weierstraß theorem can be formalised as a sentence $\mathrm{BW}$
in the language of second order arithmetic. In two papers which inaugurated the
study of reverse mathematics, \citet{Friedman1975, Friedman1976} showed that
$\mathrm{BW}$ is equivalent over $\rca_0$ to the arithmetical comprehension
scheme---the defining axiom of the system $\aca_0$.
To prove the equivalence between $\mathrm{BW}$ and $\aca_0$, one first shows
that $\aca_0$ implies BW, by formalising the usual proof of the theorem within
that system. The reversal is then accomplished by adding $\mathrm{BW}$ to the
axioms of $\rca_0$ and showing that any instance of arithmetical comprehension
is provable from $\rca_0 + \mathrm{BW}$.

One way to think about the epistemic value of reverse mathematics is that it
uncovers the resources required in ordinary mathematical reasoning: for example,
if a proof uses a compactness argument, then weak König's lemma must be amongst
the stock of axioms which the mathematician draws upon, whether explicitly or
implicitly. That nonconstructive methods, in the form of compactness, are
required to prove the completeness theorem for first-order logic tells us
something important about that theorem and its epistemic standing.
This has many ramifications for philosophical issues. \citet{Feferman1992}
points out the application of reverse mathematical methods to indispensability
arguments. Such arguments are intended to show that certain mathematical
entities, being indispensable to science, must be accorded the same ontological
rights as those whose existence is empirically confirmed. This leaves two
critical questions unanswered: which mathematical entities does this argument
show us to be committed to the existence of, and what principles concerning
those entities must we endorse in order to carry out the mathematics that is
indispensable to science?
Real analysis is a natural starting point, since our current best physical
theories model spacetime in terms of a geometrical continuum, as a type of
differentiable manifold.
By formalising our best physical theories in second order arithmetic, we could
obtain far sharper answers, by showing that theorems of analysis required for
physics are equivalent to particular systems studied in reverse mathematics.

This paper, however, is not concerned with applications of reverse mathematics
to the indispensability argument. Instead, it addresses the relationship of
reverse mathematics to the foundational views espoused by Hilbert and Weyl, and
the finitist and predicativist programs developed in their work and that of
their successors.
More generally, it explores and defends the usefulness of reverse mathematics
for determining the limits of what can be proved within foundational schemes
that can be formalised in the setting of second order arithmetic, and for
demonstrating to adherents of particular foundational approaches that they are
unable to recover a given part of mathematics within their chosen foundation.
It then turns to the task of determining whether one component of standard
reverse mathematical practice, namely a proof-theoretically weak base theory
over which equivalences between mathematical theorems and subsystems of second
order arithmetic are proved, is essential.

§\ref{sec:foundational_analysis} gives a detailed account of the motivations
and methodology that underpin the reverse mathematical analysis of foundations.
§\ref{sec:partial_hilbert} and §\ref{sec:predicativism} are then devoted to
case studies of particular foundational programs: a partial realization of
Hilbert's program due to \citet{Simpson1988}, and predicativism as initially
developed by Weyl and then extended through the work of Kreisel, Feferman,
Schütte, and others.
The paper then examines a proposal of \citet{Shore2010, Shore2013} to abandon
the standard practice of proving reverse mathematical equivalences over the base
theory $\rca_0$, and instead concern ourselves only with whether the principles
involved are true in the same Turing ideals, that is to say, $\lt$-structures
whose first order parts consist of the standard natural numbers $\omega$ and
whose second order parts are classes of sets $\mathcal{C} \subseteq
\mathcal{P}(\omega)$ closed under Turing reducibility and recursive joins. In
§\ref{sec:shore_intro} we introduce and motivate Shore's proposal, and then in
§\ref{sec:justification} we argue that in failing to respect the justificatory
structure of the foundational programs mentioned above, Shore's equivalence
relation shows itself to be inappropriate for analysing foundations for
mathematics in the way described above. Finally, in §\ref{sec:complexity} we
show that this equivalence relation is highly complex, and thus brings with it
attendant theoretical commitments that exceed those acceptable to proponents of
the type of foundational programs being analysed.

\section{Reverse mathematical analysis of foundations}
\label{sec:foundational_analysis}

One of the main philosophical roles attributed to reverse mathematics in the
current literature is what we shall call \emph{foundational analysis}. This
application has been strongly promoted by Stephen Simpson, born out of his view
(stated amongst other places in \citet{Simpson2009} and \citet{Simpson2010})
that there is a correspondence between subsystems of second order arithmetic and
foundational programs such as Weyl's predicativism and Hilbert's finitistic
reductionism.
By providing a hierarchy of comparable systems, and proving the equivalence of
theorems of ordinary mathematics to these systems, reverse mathematics
demonstrates what resources a particular theorem requires, and what theorems a
given system cannot prove. In other words, when committing to a foundational
system, reverse mathematics lets us know precisely what we are giving up.
Crucially, it also tells us when a proponent of such a system employs
mathematical resources that she is not entitled to, as they go beyond what her
preferred foundation can prove.

The following example should clarify the notion of foundational analysis.
Suppose Sarah is a predicativist in the tradition of \citet{Weyl1918}. She
believes that the natural numbers form a completed, infinite totality, and that
sets which can be defined arithmetically---i.e. with quantifiers ranging over
the natural numbers, but not over sets of them---also exist. This would lead her
to accept the arithmetical comprehension scheme, and thus the subsystem of
second order arithmetic $\aca_0$. She might even accept a somewhat stronger
system; this possibility is explored in §\ref{sec:predicativism}. But given
Sarah's predicativist outlook she would resist the thoroughly impredicative
axiom scheme of $\Pi^1_1$ comprehension, and its associated subsystem of second
order arithmetic $\pica_0$.

Now suppose that her colleague Rebecca disagrees with Sarah's predicativism and
wants to persuade her that it is an inappropriate foundation for mathematics.
She might argue as follows: Since Sarah wants her predicativist outlook to
provide a foundation for all of mathematics, it would be strange if she failed
to account for important theorems of ordinary mathematics---say, in abelian
group theory.
Consider the statement
\begin{itemize}
    \item[($\star$)]
        Every countable abelian group can be expressed as a direct sum of a divisible
        group and a reduced group.
\end{itemize}
The group theorist in the street, Rebecca argues, believes this to be true.
Sarah might tentatively agree, whereupon Rebecca would point out the following
theorem from reverse mathematics: assuming that ($\star$) holds, one can prove
(in $\rca_0$, which Sarah clearly accepts) the $\Pi^1_1$ comprehension scheme
\citep*[theorem 6.3, p.~178]{FriSimSmi1983}.

It appears that Sarah has some explaining to do. Either she must abandon her
predicativism, or she must push back against the naturalistic line Rebecca is
urging upon her. Neither course appears terribly palatable, while the fact that
this theorem is drawn not from set theory or some other area of mathematics
whose ontological commitments might be thought extravagant could be taken as
evidence that the problem here is a pressing one. The contentious statement is
an ordinary theorem from a core area of mathematics, which reverse mathematical
analysis shows us to have substantial proof-theoretic strength.

The broadly naturalistic argument that Rebecca makes to Sarah can be generalised
in a straightforward way. Let $\mathcal{F}$ be a foundation for mathematics
which accepts classical logic as leading to correct conclusions, and let
$S_\mathcal{F}$ be a subsystem of second order arithmetic containing $\rca_0$
such that the $\mathcal{F}$-theorist accepts that $S_\mathcal{F}$ is a faithful
formalisation of the principles of $\mathcal{F}$. Then any participant in the
foundational dialectic may (as Rebecca does in the example above) fill in the
following schematic argument with her favourite examples, and make it to the
$\mathcal{F}$-theorist:
\begin{quote}
    Consider the ordinary mathematical theorem $P$, which may be faithfully
    formalised in the language of second order arithmetic as the sentence
    $\varphi$. $\varphi$ is equivalent over $\rca_0$ to the subsystem of second
    order arithmetic $S(\varphi)$. But $S_\mathcal{F}$ cannot prove the axioms
    of $S(\varphi)$, and thus cannot prove $\varphi$. So $\mathcal{F}$ cannot
    recover the ordinary mathematical theorem $P$, and is thus inadequate as a
    foundation for mathematics.
\end{quote}

It is not necessary to suppose that all instances of this argument scheme will
be persuasive to all $\mathcal{F}$-theorists: foundational analysis does not
provide a knockdown argument against predicativism, or indeed any foundational
view with limited mathematical resources. Rather, it makes arguments like the
dispute between Rebecca and Sarah precise: we can see, within a common framework
(namely the base theory $\rca_0$, and the coding required to represent ordinary
mathematical concepts in it), just where the boundaries of these foundational
systems lie. As a rational agent, Sarah surely formed her foundational views in
the full understanding that they require her to give up on any mathematics that
her view deems to be unfounded. The decision to give up on or stick with her
foundation is not one to be taken lightly, and it is one that should be made by
considering the relevant facts. These facts can, in large part, be supplied by
foundational analysis, which allows Sarah and the rest of us to see precisely
what is at stake.

For foundational analysis to play a useful philosophical role in mediating
between disputants with different foundational stances, it must be possible to
carry out this analysis on ground which is common between the disputants.%
\footnote{%
For a closely related discussion, albeit one which treats much stronger logics
and axiomatic principles than those which are the subject of this article, see
\citet{Koellner2010}.
}
So while a predicativist like Sarah and a platonist like Rebecca might disagree
about whether $\Pi^1_1$ comprehension is a valid axiom, they both accept
the laws of classical logic and at least the axioms of $\rca_0$, as well as the
faithfulness of the representation of the theorem ($\star$) in second order
arithmetic, and thus both will agree that ($\star$) above is not predicatively
provable. In other words, foundational analysis makes it clear where the fault
lines lie, and the existence of common ground makes the conclusion available not
just to those who accept stronger axioms or rules of inference, but those who
are committed to a more limited foundational framework and will only accept
mathematical conclusions derived within that framework.

Notice that Sarah already accepted that $\Pi^1_1$ comprehension was not a
predicative principle, otherwise she would not have been able to deduce that
theorem ($\star$) about abelian groups was not predicatively provable. In
accepting this Sarah goes beyond what her foundation can formally prove. If she
accepts $\aca_0$ and no more, then it is difficult to see how she can separate
$\pica_0$ from $\aca_0$. $\pica_0$ implies all instances of arithmetical
comprehension, so $\aca_0$ is a subsystem of $\pica_0$, but in order to show
that it is a proper subsystem, one typically construct a model of $\aca_0$ that
is not a model of $\pica_0$. In doing so, however, one thereby proves the
consistency of $\aca_0$, which is (assuming that $\aca_0$ really is consistent)
not something that $\aca_0$ can prove. Any proof that $\pica_0$ properly extends
$\aca_0$ therefore relies on theoretical resources not available within $\aca_0$
itself. Needless to say, we cannot eliminate the assumption of the consistency
of $\aca_0$, since if $\aca_0$ is inconsistent, then it proves everything that
$\pica_0$ does, which is to say every sentence in the language of second order
arithmetic.

The upshot of this is that Sarah cannot prove that $\Pi^1_1$ comprehension is
not a predicative principle merely on the basis of her acceptance of any fixed
predicative formal theory, no matter how strong it is, since we can re-run the
above argument for any system $S$ such that $\aca_0 \subseteq S \subsetneq
\pica_0$. For Sarah or any predicativist, the judgement of the impredicativity
of $\Pi^1_1$ comprehension must therefore be justified by some other means. One
candidate justification might be Sarah's acceptance of the soundness of the
predicative formal theory $\aca_0$, or a predicative extension thereof.
Alternatively, the impredicativity of $\pica_0$ might itself be taken as a basic
(albeit defeasible) belief. The $\Pi^1_1$ comprehension scheme quantifies over
all sets of natural numbers, and appears to do so in an essential way. In the
absence of evidence to the contrary, Sarah should assume that $\pica_0$ is an
impredicative axiom system, and thus unacceptable on the basis of her
predicativist stance.
A third option is quietism about the impredicativity of $\pica_0$: Sarah could
suspend judgement about whether or not $\Pi^1_1$ comprehension is justifiable on
a predicative basis. Sarah's predicativity would thus be an entirely positive
view, as Sarah would accept any statement that can be shown to be predicatively
provable (in the Feferman--Schütte sense), but not deny that a statement is
predicative (save those which are predicatively refutable). While this is a
coherent position, and one which can be applied quite generally to many
foundational views, at least prima facie it fails to do justice to the
predicativist outlook. Predicativism was historically motivated by the apparent
vicious circularity of impredicative definitions, and the positive epistemic
view that those objects exist which can be defined by quantifying over only
objects already shown to exist is linked to the negative view that objects that
cannot be so defined do not exist. Without such a negative view in the
background, the positive program seems to lose some of its bite: while
predicative mathematics is well and good, there is little to recommend it as a
stopping point when the mathematical fruits of impredicativity are just over the
horizon.

In order to make the kind of naturalistic argument sketched above, it is not
sufficient to formalise a foundational system $\mathcal{F}$ as a subsystem of
second order arithmetic $S_\mathcal{F}$ in a way which is acceptable to the
$\mathcal{F}$-theorist. One must also ensure that the mathematical theorems
whose proof-theoretic strength is appealed to in the argument are formalised in
a faithful way. For example, there are theorems of topology which are provably
equivalent to $\Pi^1_2$ comprehension \citep{MumSim2005}. Since the formal
language of second order arithmetic only allows one to quantify over countable
objects, any representation of uncountable objects must be indirect and relies
on the availability of suitable countable codes for those objects. This
availability, and thus the faithfulness of formalisations of ordinary
mathematical notions in second order arithmetic, must be proved in some suitable
metatheory which can quantify directly over the uncountable objects in question,
and prove the existence of the countable codes of those objects.%
\footnote{%
In some cases the metatheoretic axioms required are strong: the existence of an
uncountable topological space with a countable basis was shown by
\citet{Hunter2008} to imply the axioms of $\z_2$, the full system of second
order arithmetic.
}

If Rebecca wanted to invoke these theorems in her attempt to persuade Sarah that
predicativism is inadequate to mathematical practice, and thus mathematical
truth, then her argument would appear to rely on a suppressed premise, namely
the faithfulness of the representation in second order arithmetic of these
topological spaces.
Sarah could therefore respond that the statements in the language of second
order arithmetic which Rebecca takes to be formalisations of theorems of
topology are not, from her predicative perspective, anything of the sort.
Instead they are simply $\lt$-sentences that are not predicatively provable. For
them to be formalisations of particular theorems of topology requires that they
are faithful translations of those theorems, and the proof of that faithfulness
requires theoretical resources that, being strongly impredicative, she is not
willing to commit to.

Foundational analysis in the form sketched above thus seems to impose a natural
criterion on formalisations, namely that the faithfulness of the codings used
must be provable in a conservative extension $S$ of a theory accepted by
proponents of the foundation being analysed. $S$ would be a formal version of the
metatheory discussed above, with higher type variables ranging over uncountable
sets, allowing the direct formalisation of higher type objects such as
uncountable topological spaces.
This would ensure reverse mathematical results could be read as intended, i.e.
as demonstrating the mathematical resources necessary to prove given theorems of
ordinary mathematics. It would then allow the kind of naturalistic argument
given by Rebecca to be evaluated by proponents of a given foundation, within the
theoretical framework they already accept. In the ideal case, the faithfulness
of the codings involved would be provable in a conservative extension of the
base theory, thus allowing reverse mathematical results to be evaluated by
anyone who accepts the axioms of that base theory.

In the next two sections we will study more closely two historical,
philosophically-motivated foundational programs, and their connections to
reverse mathematics and subsystems of second order arithmetic: finitism in the
sense descending from Hilbert's program in \S\ref{sec:partial_hilbert}, and
predicativism in the spirit of Weyl in \S\ref{sec:predicativism}.
%
Before doing so, it is worth remarking that the role played by foundational
analysis in the historical development of reverse mathematics is somewhat
ambiguous. While reverse mathematics has a broadly foundational aim, namely
determining the axioms necessary to prove theorems of ordinary mathematics, it
is unclear how much research in reverse mathematics itself has been directly
motivated by foundational analysis in the sense discussed here.%
\footnote{
Indeed, much of the current research in reverse mathematics is focused on other
concerns, especially the use of tools from computability theory to explore the
growing constellation of intermediate and incomparable subsystems between
$\rca_0$ and $\aca_0$ known as the Reverse Mathematics Zoo
\citep{Dzhafarov2015}. A summary of current research frontiers can be found in
\citet{Montalban2011}.
}

lines of research in the related field of reductive proof
theory that are more explicitly motivated

In contrast, work in the related field of reductive proof theory is more
explicitly motivated 
by goals related to foundational analysis, namely determining what fragment of
ordinary mathematics can be recovered in subsystems of second order arithmetic
that are proof-theoretically reducible to finitistic or constructive systems.
This line of research is known as the relativized Hilbert program. Its inception
is usually traced back to \citet{Ber1967}, and its goals, methods and results
have been articulated by, amongst others, \citet{Sieg1988} and
\citet{Feferman1988a}.

\section{Finitistic reductionism}
\label{sec:partial_hilbert}

Hilbert's program was to reduce infinitary mathematics to finitary mathematics.
He viewed finitism as a secure foundation for mathematics, free of the paradoxes
which arose from seemingly natural assumptions and normal mathematical reasoning
about infinite collections. This reduction was to be accomplished by giving a
finitary consistency proof for infinitary mathematics, which for present
purposes can be identified with $\zfc$. Hilbert thought that employing
infinitary methods in mathematics, such as assuming the existence of infinite
collections, could be viewed simply as a way to supplement our finitistic
theories with ideal statements, analogous to ideal elements in algebra. Ideal
statements are thus intended to be eliminable, at least in principle: the
purpose of Hilbert's desired consistency proof was to show that we can use
infinitary mathematics to get finitary results, and that those results are
finitistically acceptable.

Gödel's second incompleteness theorem shows that there can be no such 
consistency proof, and thus that Hilbert's program cannot be carried out in its
entirety. Many even consider Gödel's theorems to have shown that Hilbert's
program is entirely bankrupt.%
\footnote{%
For a contrary view, see \citet{Detlefsen1979}.
}
While it certainly seems to block the full realization of the enterprise,
\citet{Simpson1988} argues that the possibility of partial realizations remains.
But since the consistency proof Hilbert sought is out of reach, the latter-day
finitistic reductionist must find other ways to demonstrate that their uses of
ideal statements are in principle eliminable. Instead of trying to prove the
consistency of infinitary systems directly, finitistic reductions of infinitary
systems can be carried out in a relativised way, following the template laid
down by \citet{kreisel1968}. We now sketch how such reductions work.%
\footnote{%
An excellent survey of this topic which also details the foundational picture
behind such relativised versions of Hilbert's program is \citet{Feferman1988a}.
}

Suppose we have two theories $T_1$ (in a language $\mathcal{L}_1$) and $T_2$ (in
$\mathcal{L}_2$), both of which contain primitive recursive arithmetic. Suppose
also that we have a primitive recursive set of formulae
$\Phi \subseteq Fml_{\mathcal{L}_1} \cap Fml_{\mathcal{L}_2}$
containing every closed equation $t_1 = t_2$.
A proof-theoretic reduction of $T_1$ to $T_2$ which conserves $\Phi$ is a
partial recursive function $f$ which, given any proof from the axioms of $T_1$
of a sentence $\varphi \in \Phi$, produces a proof of $\varphi$ from the axioms
of $T_2$. If the existence of $f$ can be proved in $T_2$, it then follows that
$T_2$ proves (a formalisation of) the following conditional statement:
``If $T_2$ is consistent then $T_1$ is consistent.''
For if $T_1$ proves that $0 = 1$, then $f$ will transform any proof of $0 = 1$
in $T_1$ into a proof of $0 = 1$ in $T_2$.

If the existence of a proof-theoretic reduction of this sort can be proved in a
finitary system, then we call it a \emph{finitary reduction}. In order for a
proof-theoretic reduction $f$ from an infinitary system to a finitary one to
provide a partial realization of Hilbert's program, $f$ must be a finitary
reduction. Otherwise the result has a circular character unacceptable within a
reductionist program: it would amount to using ideal methods to show that ideal
methods are acceptable. Similarly, an infinitary proof of a conservativity
theorem is insufficient to demonstrate the reducibility of an infinitary system
to a finitary one.

If Hilbert had succeeded in providing a finitary consistency proof for
infinitary mathematics then there would have been no need to mark out the
boundary between finitary and infinitary methods with any precision, as the
proof would have made use of methods which were clearly finitary in nature.
In order to obtain the conservation results that demonstrate that certain
infinitary systems are finitistically reducible, and thereby partially realize
Hilbert's program, Simpson's route to a partial realization of Hilbert's program
requires that we formalise our conception of a finitary system. The formal
system which Simpson selects is primitive recursive arithmetic ($\pra$),
following the thesis proposed by \citet{Tait1981}. The rest of Simpson's
argument rests squarely on this identification of finitist provability with
provability in $\pra$: he does not offer any new considerations in support of
Tait's thesis, instead simply accepting it and proceeding accordingly.

Fixing $\pra$ as the finitary system to which infinitary systems must be reduced
to, the next question is which infinitary systems are finitistically reducible
to $\pra$. One such system is $\wkl_0$, the system obtained by adding weak
König's lemma
(``Every infinite subtree of $2^{<\mathbb{N}}$ has an infinite path'')
to $\rca_0$.
Friedman [1976, unpublished] used model-theor\-etic techniques to
show that $\wkl_0$ is $\Pi^0_2$ conservative over $\pra$, and thus consistent
relative to $\pra$; the proof can be found in \citet[§IX.3]{Simpson2009}.
Subsequently \citet{Sieg1985} gave a primitive recursive proof transformation
which, given a proof of a $\Pi^0_2$ theorem $\varphi$ in $\wkl_0$, generates a
proof of $\varphi$ in $\pra$. Unlike Friedman's result, this proof-theor\-etic
derivation of the conservativity theorem is itself a piece of finitistic
mathematics: it is provable within a finitary system, thus making the reduction
finitary. As the complexity of consistency statements is $\Pi^0_1$, if $\wkl_0$
proves the consistency of $\pra$ then so does $\pra$ itself. From this Simpson
concludes that $\wkl_0$ is finitistically reducible to $\pra$, and so the
fragment of mathematical reasoning which one can carry out in $\wkl_0$ is
finitistically acceptable, in the following sense. Any $\Pi^0_1$ sentence
provable in $\wkl_0$ is finitistically meaningful, in virtue of its form, but
it is also provable in $\pra$ (by the conservativity theorem), and thus
finitistically provable (by Tait's thesis). Any theorem of $\wkl_0$, such as the
Heine--Borel covering theorem or the Hahn--Banach theorem for separable spaces,
is thus legitimised as a lemma that can be invoked in order to prove a
finitistic theorem.

By referring to the reductionist project he proposes as a partial realization of
Hilbert's program, Simpson opens himself up to the criticism that his
interpretation of Hilbert is a misreading, as alleged by
\citet[p.~874]{Sieg1990}, who suggests that Simpson's project might be better
understood as a partial realization of Kronecker's views on foundations of
mathematics. While questions of historical interpretation are important, our
present purpose is not Hilbert scholarship, but determining whether there is a
defensible core to Simpson's position, and thus whether the reverse mathematics
of $\wkl_0$ make a contribution to foundational analysis. Starting from Tait's
thesis that finitist provability is provability in $\pra$, together with the
Hilbertian contention that only $\Pi^0_1$ sentences are finitistically
meaningful, the finitistically provable conservativity theorem gives us strong
prima facie reason to take Simpson's finitistic reductionism seriously. Since
the foundational analysis of finitistic reductionism can be carried out in a
base theory that is itself finitistically reducible, its results are available
to the finitist, who can thereby see that (for example) the Heine--Borel theorem
is finitistically reducible, but the Bolzano--Weierstraß theorem is not.

With the positive case in hand, we turn to potential criticisms of Simpson's
view. The first is his reliance on Tait's thesis, which has taken fire from many
quarters. Broadly speaking, such complaints fall into two camps: that $\pra$ is
too weak to encompass all of finitistic reasoning, and that it is too strong.
Those in the former camp include \citet{Kreisel1958b}, who concluded that
finitist provability coincides with provability in $\pa$. \citet{Detlefsen1979}
argued that adding instances of a restricted version of the $\omega$-rule is
also finitistically acceptable, although Detlefsen's position has in turn been
criticised, for example by \citet{Ignjatovic1994}. Two proposals that fall into
the latter camp are made by \citet{Ganea2010}. From the broad spread of
conclusions reached it is clear that what finitistic reasoning consists in is,
to say the least, disputed. However, Tait's arguments provide a robust defence
of the thesis that primitive recursive arithmetic demarcates the limits of
finitistic reasoning, and moreover, one that has gained wide acceptance. We
therefore conclude that on the one hand, extant arguments against Tait's thesis
entail that we should not consider Simpson's identification of finitistic
reducibility with proof-theoretic reducibility to $\pra$ to be established; but
on the other, since a strong case can be made in favour of the thesis, Simpson's
finitistic reductionism should be taken seriously as a foundation of
mathematics.

\citet{Bur2010} criticises the finitistic reducibility of $\wkl_0$ from another
direction, arguing that the analysis leading to the identification of finitistic
provability with provability in the formal system $\pra$ cannot be carried out
from a finitistic point of view. This means that the conservativity theorem does
not, by itself, justify the finitist in believing any $\Pi^0_1$ sentence
provable in $\wkl_0$. At best, it provides a recipe for producing $\pra$ proofs
from $\wkl_0$ proofs, which the finitist must then verify by assuring themselves
that each of the axioms of $\pra$ used in the proof is in fact finitistically
acceptable.
Burgess offers the following way out for the finitistic reductionist (p.~139):
by limiting the induction principle in $\rca_0$ and $\wkl_0$ we can define
subtheories $\rca_*$ and $\wkl_*$, in which $\Sigma^0_1$ induction is replaced
by $\Sigma^0_0$ induction plus a sentence asserting that the exponential
function is total.%
\footnote{%
The theories $\rca_*$ and $\wkl_*$ are often referred to in the reverse
mathematics literature as $\rca_0^*$ and $\wkl_0^*$, for example by
\citet{SimSmi1986}, who first isolated these systems.
However, this notation is confusing because in most other cases in reverse
mathematics, the superscript is used to refer to additional set existence axioms
adjoined to the theory, as in the case of $\wkl_0^+$, $\aca_0^+$, and so on,
while the subscript is used to indicate a restricted induction axiom ($\aca$
versus $\aca_0$, for example). We therefore use \citet{Montalban2011}'s
convention and refer to the system defined by the basic axioms, the recursive
comprehension scheme, and induction for $\Sigma^0_0$ formulas plus the totality
of exponentiation, as $\rca_*$, and the system obtained by adding weak König's
lemma to $\rca_*$ as $\wkl_*$.
}
$\wkl_*$ is conservative over a proper subtheory of $\pra$, known as
$\mathrm{I}\Delta_0 + \mathrm{exp}$, and since provability in this system does
not press up against the bounds of finitist provability, the finitist can
recognise that all proofs in this system are finitistically acceptable
\citep{SimSmi1986}.
Many, albeit not all, of the theorems of ordinary mathematics that are provable
in $\wkl_0$ are also provable in $\wkl_*$. For $\Pi^0_1$ sentences provable in
$\wkl_*$ the finitist can therefore work in the infinitary system without
needing to check that the resulting proof in $\pra$ uses only finitistically
acceptable principles, since this is already guaranteed by the conservativity
theorem---which, by a result of \citet{Sie1991}, is finitistically provable.


A further objection is that Simpson's argument does not in any way pick out
$\wkl_0$ as the unique formal counterpart of the program of finitistic
reductionism. \citet{BroSim1993} present a system they call $\wkl_0^+$, which
extends $\wkl_0$ with a strong formal version $\mathrm{BCT}$ of the Baire
Category Theorem. They prove, using a forcing argument, that $\wkl_0^+$ is
$\Pi^1_1$ conservative over $\rca_0$. It follows from a result of
\citet{Parsons1970} that $\wkl_0^+$ is $\Pi^0_2$ conservative over $\pra$, and
that this conservativity theorem can be proved in $\pra$ itself.%
\footnote{%
\citet{Avigad1996} showed how the forcing arguments of \citet{BroSim1993} and
Harrington could be formalised in the base theory $\rca_0$, thus giving a new
effective proof of the $\Pi^1_1$ conservativity of $\wkl_0^+$ over $\rca_0$,
with only a polynomial increase in the length of proofs.
}
So while $\wkl_0$ is, modulo Tait's thesis, a finitistically reducible system,
it is but one partial realization of Hilbert's program. $\wkl_0^+$ is
demonstrably another, and indeed a stronger one, since it satisfies the same
criteria of finitistic reducibility whilst properly extending $\wkl_0$.
One might think that this undermines Simpson's claim that the Big Five
subsystems of second order arithmetic correspond to existing foundational
programs, but this is not a fair reading of Simpson's position: he does not
claim that these systems are the unique formal correlates of these foundational
approaches. It is consistent with his position that there are a variety of
infinitary yet finitistically reducible systems. Nevertheless, it is weak
König's lemma that has been found equivalent to many theorems of ordinary
mathematics, not $\mathrm{BCT}$. This is evidence for the (defeasible) claim
that $\wkl_0$ is a mathematically natural stopping point in a way that
$\wkl_0^+$ is not. $\wkl_0^+$ is finitistically reducible just as $\wkl_0$ is,
while being a proper extension of it, so mathematically natural stopping points
do not appear to always align cleanly with justificatory stopping points---or if
they do, then we have not yet identified the sources of justification of these
axiom systems in a sufficiently fine-grained way.

\citet{PatYok2016} have shown that the statement known as Ramsey's theorem for
pairs and two colours, $\mathsf{RT}^2_2$, is finitistically reducible. Moreover,
by proving an amalgamation theorem, they also show that $\wkl_0 +
\mathsf{RT}^2_2$ is finitistically reducible. Since $\mathsf{RT}^2_2$ is
incomparable with $\wkl_0$ \citep{Jockusch1972, Liu2012}, this is a substantial
extension of the principles of proof known to be finitistically reducible, which
now includes a large class of combinatorial and model-theoretic principles that
have been the focus of much of recent research in reverse mathematics (for a
survey of these principles, see \S4--5 of \citet{Shore2010}; a more comprehensive
introduction to the study of combinatorial principles related to Ramsey's
theorem is \citet{Hirschfeldt2014}).

Having considered the extent to which the reverse mathematics of systems such as
$\wkl_0$ provides a foundational analysis of the program of finitistic
reductionism, and the relationship of this reductive program to Hilbert's, we
now turn to the study of predicativism in the spirit of Weyl and its connections
to subsystems of second order arithmetic such as $\aca_0$ and $\atr_0$.

\section{Predicativism and predicative reductionism}
\label{sec:predicativism}

Predicativism is the view that only those sets that can be defined without
reference to themselves are legitimate, existing objects. Predicativism given
the natural numbers is the view that the natural numbers $\Nat$ form a
completed, infinite totality and thus quantification over the natural numbers is
a legitimate way to define sets. For the rest of this article, whenever
`predicativism' and its cognates are invoked, it is predicativism given the
natural numbers that is meant. Predicativism can be seen as a middle ground
between finitism, in which only finite entities are accorded real existence, and
forms of set-theoretic platonism, where defining objects by impredicative
quantification is acceptable because the objects being quantified over are
considered to have an existence independent of their definition.
Although the predicativist only accepts definitions that quantify over objects
that are not themselves the subjects of the definition, collections introduced
by previous predicative definitions are legitimate objects that one may quantify
over when defining new objects. Predicative definability is thus an iterable
notion. Given the natural numbers $\Nat$ we may define the collection $R_1$ of
sets definable using only the language of arithmetic and quantifiers ranging
over the natural numbers (in other words, the arithmetical sets). We may then
take the collection of sets definable by formulas in the second-order language
of arithmetic, but where the second order quantifiers are relativized to range
over $R_1$. Since the second order quantifiers do not range over the objects
being defined (as they are restricted to ranging over objects we have previously
guaranteed the existence of), this gives us an expanded collection of
predicatively definable sets of natural numbers $R_2$. The natural next step is
to form a ramified hierarchy of sets of natural numbers. Given a domain
$D \subseteq \mathcal{P}(\omega)$, let $D^*$ be the collection of sets of
natural numbers defined by formulas in the language $\lt$ of second order
arithmetic, $\varphi^D$, where the second order quantifiers in $\varphi$ are
relativized to range over $D$. We then define the ramified analytical hierarchy
by transfinite recursion on ordinals as
\begin{align*}
    R_0            &= \emptyset \\
    R_{\alpha + 1} &= (R_\alpha)^* \\
    R_\lambda      &= \bigcup_{\beta < \lambda} R_\beta
                      \text{ for limit $\lambda$}.
\end{align*}

The ramified theory of types, developed by Russell and Whitehead in
\emph{Principia Mathematica} \citep{WhiRus1925}, proved too cumbersome for real
use in mathematics. Hermann Weyl's predicative development of analysis in
\emph{Das Kontinuum} \citep{Weyl1918} showed that a theory of the arithmetical
sets $R_1$ is already sufficient to develop a substantial portion of classical
analysis, including a sequential form of the least upper bound principle. A
modern formal reconstruction of arithmetical analysis in the mode of Weyl is
given by the system $\aca_0$, defined as $\rca_0$ plus the arithmetical
comprehension axiom: every set definable by an arithmetical formula exists.%
\footnote{%
Here we brush over the details of the connection between Weyl's system and
subsystems of second order arithmetic, for which \citet{Feferman1988} is the
authoritative source. An accessible summary of Weyl's development of
arithmetical analysis can be found in \citet{Feferman2005}.
}
Weyl showed that the amount of mathematics one could develop in his predicative
framework was extensive, and allowed one to recover much of classical analysis.
The reverse mathematics of $\aca_0$ can be viewed as a continuation of this
project, showing us not only what mathematics can be predicatively proved, but
also which theorems cannot be proved in $\rca_0$ or $\wkl_0$, and actually
require arithmetical comprehension. This includes theorems in analysis such as
the Cauchy convergence theorem and the Ascoli lemma, but also theorems from
algebra and combinatorics: that every countable vector space over $\mathbb{Q}$
has a basis, and Ramsey's theorem that for every $k \in \Nat$, every colouring
of $[\Nat]^k$ has a homogeneous set.

Unlike $\wkl_0$, which is only reducible to a finitistic system, and not a
finitistic system in and of itself, $\aca_0$ is a formal system whose axioms can
all be directly justified on predicative grounds. No general, limitative account
of what principles are predicatively acceptable is required in order to show
that $\aca_0$ is a predicative system.
As indicated above, however, predicative definability is an iterable notion, and
consequently Weyl accepted a Principle of Iteration that, as
\citet{Feferman1988}'s analysis demonstrates, goes beyond what his restriction
to the arithmetically definable sets allows.%
\footnote{%
See §7 of \citet{Weyl1918} for a definition of the Principle of Iteration, and
the discussion in \citet{Feferman1988}, especially that on pp.~264--5 of the
revised version in \citet{Feferman1998}.
}
The strength of Weyl's system with the Principle of Iteration therefore exceeds
that of $\aca_0$, as sketched in §8 of \citet{Feferman1988}, although
\citet{FefJag1993, FefJag1996} proved that it is still a predicative system, in
a sense we will now explore.

The iterability of predicative definability suggests that there should be a
correspondingly iterable notion of \emph{predicative provability}, and it is
this notion that is the subject of Feferman and Schütte's influential analysis
of the limits of predicativity. The ramified analytical hierarchy provides a
standard model on which to base the development of predicative theories,
starting with the language. Supplementing the usual first order language of
arithmetic, the \emph{language of ramified analysis} includes a stock of set
variables $X^\beta, Y^\beta, Z^\beta, \dotsc$ for each recursive ordinal
$\beta$. Iterated predicative definability is formalised by a transfinite
progression of formal systems of ramified analysis $\mathbf{RA}_\alpha$. Each
such system has the ramified comprehension scheme
\begin{equation*}
    \exists{X^\beta}\forall{n}( n \in X^\beta \leftrightarrow \varphi(n))
\end{equation*}
for all $\beta \leq \alpha$ and all formulas $\varphi(n)$ in the language of
ramified analysis such that the bound and free set variables in the formula all
have ordinal indices smaller than $\beta$. It also has the following limit rule,
where for all (codes for) limit ordinals $\lambda \leq \alpha$, and each formula
of ramified analysis $\psi(X^\lambda)$ with just $X^\lambda$ free,
if $\psi(X^0), \dotsc, \psi(X^\beta), \dotsc$ for all $\beta < \lambda$, then
$\psi(X^\lambda)$ also holds.

This definition just leaves open how far the iteration process can go and still
be considered predicative. Suspicion naturally attaches to the ordinals indexing
the theories $\mathbf{RA}_\alpha$, for two reasons. Firstly, the question of
whether a recursive linear order $\prec$ is a wellordering is, in general,
impredicative: the statement that $\prec$ is wellordered quantifies over all
sets of natural numbers. Secondly, in the language of ramified analysis there is
no single formula that expresses that $\prec$ is wellordered. As far as this
second issue is concerned, let $\mathrm{WO}^\beta(\prec)$ be the statement that
no set $X^\beta$ codes an infinite descending sequence in the linear order
$\prec$. Then, by the limit rule, any proof in a system of ramified analysis
that $\mathrm{WO}^0(\prec)$ can be lifted to a proof that
$\mathrm{WO}^\beta(\prec)$ for each new $\beta$ introduced. The first issue can
then be resolved by the introduction of the notion of a \emph{predicative
ordinal}. The predicative ordinals are those recursive ordinals which can be
predicatively proved to be wellordered, and the predicatively acceptable systems
of ramified analysis are those indexed by predicative ordinals. This might seem
circular, since we are using the notion of a predicative ordinal in order to
characterise predicative provability, and predicative provability in order to
determine which are the predicative ordinals. But as the following
\emph{autonomy condition} should make clear, this is not the case.
$0$ is a predicative ordinal.
Suppose $\prec$ is a recursive linear order and $\alpha$ is a predicative
ordinal.
If $\mathbf{RA}_\alpha \vdash \mathrm{WO}^0(\prec)$, then the ordinal
$\beta = \mathrm{ot}(\prec)$ is predicative.
This allows us to define predicative provability in terms of autonomous
transfinite progressions of systems of ramified analysis, as follows: a sentence
$\varphi$ of ramified analysis is predicatively provable if there is a
predicative ordinal $\delta$ such that $\mathbf{RA}_\delta \vdash \varphi$.
On the basis of these definitions, \citet{Feferman1964} and \citet{Schutte1964,
Schutte1965} determined the limit of predicativity, namely the least ordinal
that cannot be  proved to be wellordered within an autonomous progression of
systems of ramified analysis. This is the ordinal $\Gamma_0$, also known as the
Feferman--Schütte ordinal, or the ordinal of predicativity.%
\footnote{%
For a more detailed summary and historical background on predicativism,
predicative provability, and predicative reductionism, we refer the reader to
\citet{Feferman2005}.
Complete proofs of Feferman and Schütte's key results can be found in
\citet[ch.~VIII]{Schutte1977} or \citet[ch.~8]{Poh2009}.
}

With the limits of predicativity characterised in this manner,%
\footnote{%
Notwithstanding \citet{Wea2009}'s claim that $\Gamma_0$, as well as larger
recursive ordinals such as the small Veblen ordinal, can be proved to be
wellfounded using only predicative means.
}
we might now ask how much of second order arithmetic can be justified on
predicative grounds. Reducing subsystems of second order arithmetic to
predicative systems has practical benefits for the predicativist, since ramified
systems are so intractable in terms of the actual pursuit of mathematics. It is
also valuable from the point of view of foundational analysis, since it allows
one to determine the amount of mathematics recoverable in a predicative
framework.

A formal system $T$ is \emph{predicatively reducible} if there is an
$\alpha < \Gamma_0$ such that $T$ is proof-theoretically reducible to
$\mathbf{RA}_\alpha$, and \emph{locally predicatively reducible} if it is
proof-theoretically reducible to $\mathbf{RA}_{\Gamma_0} =
\bigcup_{\alpha < \Gamma_0} \mathbf{RA}_\alpha$, where proof-theoretic
reducibility is defined as in §\ref{sec:partial_hilbert}.%
\footnote{%
In his survey of predicativity, \citet{Feferman2005} prefers the terms
\emph{predicatively justifiable} and \emph{locally predicatively justifiable} to
\emph{predicatively reducible} and \emph{locally predicatively reducible}. This
paper sticks to the older terminology.
}
Predicative reductionism thus offers a reductionist program similar in spirit
to the finitistic reductionism discussed in §\ref{sec:partial_hilbert}.%
\footnote{%
The comparison between finitistic reductionism and predicative reductionism has
been considered explicitly in these terms by \citet[§5]{Simpson1985}.
}
By a theorem of \citet*{FriMcASim1982}, the subsystem of second order arithmetic
known as $\atr_0$ is locally predicatively reducible, with proof-theoretic
ordinal $\Gamma_0$. Moreover, it is conservative over $\mathbf{RA}_{\Gamma_0}$
for arithmetical sentences, and even for $\Pi^1_1$ sentences.%
\footnote{
Strictly speaking there is no common notion of a $\Pi^1_1$ statement shared
between the language of second order arithmetic and the language of ramified
analysis. However, given a $\Pi^1_1$ sentence $\varphi$ in the language of
second order arithmetic with bound set variables $X_1, \dotsc, X_n$, such that
$\varphi$ is provable in $\atr_0$, there exists $\alpha < \Gamma_0$ such that
the translation $\varphi^*$ of $\varphi$ into the language of ramified analysis
is provable in $\mathbf{RA}_\alpha$, where $\varphi^*$ is obtained by replacing
each $X_k$ for $1 \leq k \leq n$ with $X^0_k$.
}
So not only does $\atr_0$ agree with the predicative part of ramified analysis
about arithmetical truth, it also proves the same theorems about the
arithmetical properties of all real numbers.

The formal system $\atr_0$ consists of $\aca_0$ plus a scheme of
\emph{arithmetical transfinite recursion}. This states that the arithmetical
operations can be iterated, starting from any set $X \subseteq \mathbb{N}$,
along any countable wellordering; a full formal definition can be found in
\citet[§V.2]{Simpson2009}.
$\atr_0$ is a significant strengthening of $\aca_0$, taking us from classical
analysis to parts of descriptive set theory: arithmetical transfinite recursion
is equivalent over $\rca_0$ to the perfect set theorem (every uncountable closed
set has a perfect subset), Lusin's separation theorem (any two disjoint analytic
sets can be separated by a Borel set), and a number of statements concerning
ordinals, for example that any two countable wellorderings are comparable.

\citet[p.~140]{Bur2010}'s caution about conservativity applies to the case of
$\atr_0$ and predicative provability just as it does to the case of $\wkl_0$ and
finitist provability. The conservativity theorem in this case will, for a proof
$p$ in $\atr_0$ of a $\Pi^1_1$ statement $\varphi$, provide a primitive
recursive function $f$ that transforms $\atr_0$ proofs of $\Pi^1_1$ sentences
into proofs in $\mathbf{RA}_{\Gamma_0}$, so that $f(p)$ is a proof in
$\mathbf{RA}_\alpha$ of $\varphi^*$, for some $\alpha < \Gamma_0$. Nevertheless,
the predicative mathematician cannot immediately conclude that $\varphi^*$ is
predicatively provable, because the Feferman--Schütte analysis of the limits of
predicativity is external to the predicativist standpoint, and thus not
something the predicativist has access to: they cannot, from a predicative
standpoint, prove that all ordinals below $\Gamma_0$ are wellfounded, but can
only verify of particular presentations of ordinals below $\Gamma_0$ that they
do indeed code ordinals. To recognise that $\varphi^*$ is indeed predicatively
provable, the predicativist must first verify that $\alpha$ is a predicative
ordinal, by carrying out the bootstrapped process of proving linear orderings to
be wellfounded within predicative systems of ramified analysis described by the
analysis of predicativity.%
\footnote{%
\citet[p.~140]{Bur2010} writes that
``If we move up to the level of predicativism, the result on the
conservativeness of the system called $\atr_0$ over the system called
$\mathbf{IR}$ has the same character as the result on the conservativeness of
$\wkl_0$ over $\pra$.''
In a similar vein, \citet[p.~154]{Simpson1985} writes that
``\citet{Feferman1964} has argued successfully that his formal system
$\mathbf{IR}$ and others like it constitute a precise explication of predicative
provability.''
The system $\mathbf{IR}$ was introduced by \citet{Feferman1964}, and is
characterised by the $\Delta^1_1$ comprehension rule and the Bar Rule. Strictly
speaking, the quoted remarks are incorrect, since $\mathbf{IR}$ is a not a
predicative system, but instead a locally predicatively reducible one, just as
$\atr_0$ is. Burgess's point about conservativity, properly reformulated in
terms of conservativity over $\mathbf{RA}_{\Gamma_0}$, thus applies to
$\mathbf{IR}$ just as it applies to $\atr_0$.
}

Paralleling the response Burgess sketches on behalf of the proponent of
finitistic reductionism (discussed in §\ref{sec:partial_hilbert}), one might
think that (globally, not merely locally) predicatively reducible subsystems of
second order arithmetic offer a way to gain the benefits of predicative
reductionism without running into the problems faced by locally predicatively
reducible systems such as $\atr_0$. Unfortunately the advantages that working
within predicatively reducible subsystems of second order arithmetic offer over
working with the predicative system $\aca_0$ are minimal: as
\citet[§5]{Simpson1985} stresses, there are few theorems of ordinary
mathematics that are known to be true in the hyperarithmetic sets $\hyp$ (when
viewed as an $\omega$-model) that are not already true in the arithmetical sets
$\arith$. This is salient for predicative provability because $\hyp = 
R_{\omega_1^\mathrm{CK}}$, the sets definable by iterating the ramified
analytical hierarchy up to $\omega_1^\mathrm{CK}$. $\hyp$ is therefore an
extension of the standard model $R_{\Gamma_0}$ of ramified analysis up to
$\Gamma_0$, so if an $\lt$-sentence $\varphi$ is false in $\hyp$, then it is not
predicatively provable.

Simpson argues that the many theorems which provable in $\atr_0$ but which do
not hold in the $\omega$-model $\hyp$ demonstrate the benefits of predicative
reductionism, but not being predicatively provable, these theorems have only
instrumental value to the predicativist: statements such as the perfect set
theorem, or the theorem that any two countable wellorderings are comparable can
be used to prove predicative theorems, but are not themselves predicative.
All this goes to show that the mathematics that can be recovered by reductionist
programs such as finitistic reductionism and predicative reductionism cannot be
simply read off from reverse mathematical results, based on identifications
(such as that of $\wkl_0$ for finitistic reductionism, or $\atr_0$ for
predicative reductionism) of subsystems of second order arithmetic associated
with that foundational framework.


\section{Shore's program}
\label{sec:shore_intro}

However much it may borrow from other areas of mathematical logic, reverse
mathematics is ultimately a proof-theor\-etic endeavour. Given a theorem of
ordinary mathematics, the reverse mathematician seeks to find a subsystem of
$\z_2$ that is equivalent over a weak base theory to the theorem concerned. She
thereby finds the proof-theor\-etic strength of the theorem. Rooted in niceties
of formal systems such as axiom schemes and complexity hierarchies of formulae,
this approach may seem awkward and even unnatural to mathematicians in more
mainstream fields. As number theorist Barry Mazur says
\citep[p.~224, emphasis in original]{mazur2008},
\begin{quote}
    when it comes to a crisis of rigorous argument, the open secret is that, for
    the most part, mathematicians who are not focussed on the architecture of
    formal systems per se, mathematicians who are \emph{consumers} rather than
    \emph{providers}, somehow achieve a sense of utterly firm conviction in
    their mathematical doings, without actually going through the exercise of
    translating their particular argumentation into a brand-name formal system.
\end{quote}
Turning to the specific case of the strength of mathematical theorems,
\citet[p.~381]{Shore2010} contends that most mathematicians do not approach this
task from the viewpoint of reverse mathematics:
\begin{quote}
    While they may concern themselves with (or attempt to avoid) the axiom of
    choice or transfinite recursion, they certainly do not think about (nor
    care), for example, how much induction is used in any particular proof.
\end{quote}
Shore goes on to argue that adopting a computational approach to reverse
mathematics would solve this exegetical problem, providing a natural way for
mathematicians to understand the motivations and results of reverse mathematics.

A computational account of reverse mathematics can be considered plausible only
if mathematical principles have computational content. At least in the case of
arithmetic it is clear that this is true, as demonstrated by the pioneering
results of Gödel, Church, Turing, Post, Kleene and Rosser in the 1930s.
Computability theory holds an important status in reverse mathematics, both in
virtue of its relationship to subsystems of reverse mathematics and because it
provides a battery of tools for proving reverse mathematical results. It is
these principles and techniques which Shore appeals to when constructing his
account of computational reverse mathematics.

In particular, the major subsystems of second order arithmetic correspond to
principles from computability theory. As well as shedding light on the model
theory of these systems, these connections give us the basis for Shore's
computational reverse mathematics. The foundation of these correspondences lies
in the notion of an \emph{$\omega$-model}. These are $\lt$-structures whose
first order parts consists of the standard natural numbers
$\omega = \Set{ 0, 1, 2, \dotsc }$, and whose arithmetical vocabulary is
interpreted in the standard way, with a second order part
$\mathcal{S} \subseteq \mathcal{P}(\omega)$. $\omega$-models are thus uniquely
distinguished by their second order parts, and which sentences of $\lt$ a given
$\omega$-model $M$ satisfies is determined entirely by the sets in its second
order part. In the rest of the paper we shall therefore allow ourselves some
sloppiness and identify $\omega$-models with their second order parts whenever
no ambiguity is possible.

Since $\omega$-models of $\rca_0$ satisfy recursive comprehension, they are
closed under Turing reducibility: given an $\omega$-model $\mathcal{S}$ of
$\rca_0$, if $X \in \mathcal{S}$ and $Y \Tleq X$, then $X \in \mathcal{S}$. They
are also closed under recursive joins: if $X, Y \in \mathcal{S}$, then
$X \oplus Y \in \mathcal{S}$, where the recursive join operation is defined as
\begin{equation}
    X \oplus Y =
        \Set{ 2x | x \in X }
        \cup
        \Set{ 2y + 1 | y \in Y }.
\end{equation}
Subsets of $\mathcal{P}(\omega)$ which are closed under Turing reducibility and
recursive joins are known as \emph{Turing ideals}, and the $\omega$-models of
$\rca_0$ are precisely the Turing ideals.

Similar closure conditions apply to the $\omega$-models of the other main
subsystems of second order arithmetic. $\omega$-models of $\aca_0$ are Turing
ideals, since $\rca_0$ is a subtheory of $\aca_0$, but these models are also
closed under the Turing jump operator, while those of $\pica_0$ are closed under
the hyperjump. Computability-theor\-etic closure conditions also characterise
the $\omega$-models of the intermediate systems $\wkl_0$ and $\atr_0$. The
$\omega$-models of $\wkl_0$ are related to the Jockush--Soare low basis theorem
\citep{JocSoa1972}. The $\omega$-models of $\atr_0$ are closed under
hyperarithmetic reducibility, although there are some subtleties here; see
§VIII.4 and §VIII.6 of \citet{Simpson2009}. The Big Five thus correspond closely
to a hierarchy of computational principles of increasing power.

Shore proposes a new approach to reverse mathematics based on taking these
com\-put\-abil\-ity-theoretic characterisations of the $\omega$-models of
subsystems of $\z_2$ at face value as measuring the complexity of the theorems
equivalent to those systems. In place of the usual relations employed in reverse
mathematics, namely provability and logical equivalence over a weak base theory,
he offers the notions of \emph{computable entailment} and \emph{computable
equivalence}.

\begin{dfn}
    \label{dfn:comp_entail_equiv}
    Let $\mathcal{C}$ be a Turing ideal, and let $\varphi$ be a sentence of
    second order arithmetic. $\mathcal{C}$ \emph{computably satisfies} $\varphi$
    if $\varphi$ is true in the $\omega$-model whose second order part consists
    of $\mathcal{C}$. A sentence $\psi$ \emph{computably entails} $\varphi$,
    $\psi \models_c \varphi$, if every Turing ideal $\mathcal{C}$ computably
    satisfying $\psi$ also computably satisfies $\varphi$. Two sentences $\psi$
    and $\varphi$ are \emph{computably equivalent}, $\psi \equiv_c \varphi$, if
    each computably entails the other.
    These definitions extend to theories in the standard way.
\end{dfn}

Computable entailment removes any need for an explicit base theory: this role is
instead played by the restriction of the class of models under consideration to
$\omega$-models whose second order parts are Turing ideals. The $\omega$-models
of $\rca_0$ are precisely those models, so the base theory has not disappeared
entirely, but manifested itself in a different way, by being baked into the
definition of the computable entailment relation. Since not all $\lt$-structures
are $\omega$-models, failures of computable entailment are stronger than
failures of logical implication over $\rca_0$, since the former entails the
latter, but not vice versa. Conversely, computable entailment is weaker than
logical implication over $\rca_0$.
By the Henkin--Orey completeness theorem for $\omega$-logic \citep{Henkin1954,
Orey1956}, the computable entailment relation is extensionally equivalent to
allowing unrestricted use of the $\omega$-rule in $\rca_0$. The $\omega$-rule is
an infinitary rule of inference that, from the infinite set of premises
$\varphi(0), \varphi(1), \dotsc, \varphi(\overline{n}), \dotsc$ for all numerals
$\overline{n}$, one may infer the universal statement $\forall{n}\varphi(n)$.
Proofs using the $\omega$-rule can be represented by wellfounded, countably
branching trees.
Second order arithmetic with the $\omega$-rule is complete for $\Pi^1_1$
sentences, but not for $\Sigma^1_1$ sentences \citep{Rosser1937}.

Shore puts forward a number of considerations in support of his proposal.

\begin{itemize}
    \item[(1)] Computational reverse mathematics unites proofs of implications
        (and hence equivalences) with proofs of nonimplications (and hence
        inequivalences).
    \item[(2)] Computational reverse mathematics offers a more direct route to
        the complexity measures we take to underpin the identification of the
        strength of theorems of ordinary mathematics.
    \item[(3)] Computational mathematics is a more natural framework for
        ordinary mathematicians.
    \item[(4)] Computational reverse mathematics provides a way to deal directly
        with uncountable structures and extend reverse mathematics to the study
        of theorems concerning essentially uncountable structures.
\end{itemize}

These considerations can be understood as arguments for two quite different
conclusions. The first is that computable entailment is an important
reducibility notion with intrinsic mathematical interest, one which merits and
will reward further study.%
\footnote{%
A comparable notion that provides a finer-grained degree structure is Weihrauch
reducibility, recently taken up in the reverse mathematics context by
\citet*{DorDzhHir2015}, but a comparison of Weihrauch reducibility to
computational reverse mathematics is beyond the scope of this paper.
}
The second is that it should be the primary tool we use to carry out the general
task of reverse mathematics, namely to show what mathematical resources are
needed to prove ordinary mathematical theorems. Shore does not explicitly
endorse either of these options; neither does he propose computational reverse
mathematics as a framework in which to carry out foundational analysis. However,
by reviewing the reasons underlying his proposal, it becomes clear that
arguments based on these considerations make adopting computational reverse
mathematics as a first-class replacement for classical reverse mathematics a
serious option. This being the case, it seems reasonable to wonder whether his
framework can contribute to the analysis of foundational programs just as
classical reverse mathematics does. Before attempting to answer this question,
we first analyse in detail Shore's considerations in favour of computational
reverse mathematics.

\subsection{Unity of reverse mathematics}
\label{ssec:unity_rm}

The first consideration in favour of computational reverse mathematics is that
(1) computable entailment unites proofs of implications with the existing
practice of using computability-theoretic tools to construct Turing ideals
witnessing the failures of implications, bringing a unity to the methods of
proof in reverse mathematics. The procedure is particularly straightforward
when the sentences in question are $\Pi^1_2$, where the following template
applies. Given $\Pi^1_2$ statements
$\Phi \equiv \forall{X}\exists{Y}\varphi(X, Y)$ and
$\Psi \equiv \forall{X}\exists{Y}\psi(X, Y)$,
one constructs a Turing ideal $\mathcal{C}$ where for every $X \in
\mathcal{C}$, there exists an $Y \in \mathcal{C}$ such that $\varphi(X, Y)$,
but there is no $Y \in \mathcal{C}$ such that $\psi(X, Y)$. $\mathcal{C}$ is
therefore the second-order part of an $\omega$-model that satisfies $\rca_0$ and
$\Phi$, but not $\Psi$, and so $\rca_0$ does not prove that $\Phi$ implies
$\Psi$.
Since many important mathematical theorems are $\Pi^1_2$, the technique is
widely applicable. For instance, to show that weak König's lemma does not imply
arithmetical comprehension, we use the Jockush--Soare low basis theorem to prove
the existence of an $\omega$-model $M$ of $\wkl_0$ in which all sets are low.
Such a model will not contain $0^\prime$, and thus $M \not\models \aca_0$, since
$\aca_0$ proves the existence of the Turing jump.

It is clear how Shore's observation leads to his definition of the computable
entailment relation, but to use it to support the adoption of computational
reverse mathematics as the preferred general methodology in reverse mathematics,
we must formulate a more explicit argument. The following reconstruction seems
reasonable: that (1a) unity in methods of proof in reverse mathematics is a
theoretical virtue, and that (1b) computational reverse mathematics possesses
this unity in greater degree than classical reverse mathematics. Granting (1a)
is compatible with there being other theoretical virtues, such as tractability,
not requiring strong background assumptions, and so on. We stipulate for the
sake of argument that (1c) the other theoretical virtues of computational
reverse mathematics are at least as great as those of classical reverse
mathematics. From (1a--c), it follows that computational reverse mathematics is
a more theoretically virtuous framework than classical reverse mathematics in
which to carry out reverse mathematics.

In the next two sections we shall discover that we have reason to doubt (1c),
but for now let us grant it and concentrate on the more immediately problematic
premise (1b). By the soundness theorem for first order logic, from exhibiting a
model of $\rca_0 + \Phi + \neg\Psi$ we can infer that there is no proof in
$\rca_0$ of $\Phi \rightarrow \Psi$.
This shows that there is a unity in the methods of proof of classical reverse
mathematics: to demonstrate an implication we prove that the formal system
$\rca_0$ proves it, and to demonstrate a nonimplication we prove that the formal
system $\rca_0$ does not prove the relevant implication. The soundness and
completeness theorems for first order logic unite proof-theoretic and model-%
theoretic methods. If computational reverse mathematics has a greater unity of
methods of proof, it can only be in a methodological sense, since as Shore
points out, in practice the vast majority of proofs of nonimplications are
carried out by constructing Turing ideals.
But if the unity of methods of proof we are concerned with is merely
methodological, concerning how practitioners happen to prove nonimplications in
classical reverse mathematics, then premise (1a) starts to look shaky, since it
seems entirely reasonable to take the unity of methods of proof (in this
methodological sense) to be a purely instrumental virtue of a framework for
reverse mathematics, and not a theoretical one. Even if we grant that
computational reverse mathematics is better, from an instrumental point of view,
than classical reverse mathematics as a framework for reverse mathematics, this
seems like a minor reason to adopt it when weightier theoretical considerations
are on the table.

\subsection{Complexity and difficulty of proof}
\label{ssec:complexity_difficulty}

A more substantial motivation can be found in Shore's suggestive remarks about
the relationship between degrees of computability and methods of proof. The
essential point is as follows. Proving a theorem that can be stated in the form
``For all $X$ such that \ldots, there exists a $Y$ such that \ldots''---that is
to say, as a $\Pi^1_n$ sentence---can be understood as providing, given as input
a countably representable mathematical structure $\mathcal{A}$, a function or
relation $\mathcal{F}$ on $\mathcal{A}$. As we have seen, given some
$\mathcal{A}$, a corresponding $\mathcal{F}$ can be more or less complex,
depending on the theorem. In classical reverse mathematics, the proof-theoretic
strength of the theorem depends on where such $\mathcal{F}$ can be found in the
hierarchy of Turing degrees relativized to $\mathcal{A}$. Shore proposes that we
give a direct formulation of this complexity measure---the difficulty of
computing a solution (the function or relation $\mathcal{F}$) given a problem
(the structure $\mathcal{A}$)---rather than mediating it through first order
logical theories. Doing so ``formalizes the intuition that `being harder to
prove' means `harder to compute''' \citep[p,~153]{Shore2013}. Taking this
underlying intuition into account, the appeal of Shore's approach becomes more
clear: since computation across nonstandard models of arithmetic is highly
non-absolute, restricting the interpretation of arithmetical vocabulary to the
standard natural numbers $\omega$ allows us to fix an single underlying notion
of computation, and thus the Turing reducibility relation $\Tleq$ that allows us
to compare the complexity of solutions provided by theorems directly. This is
not available within the classical reverse mathematics framework where the base
theory $\rca_0$ provides at best a weak constraint on the models of the
arithmetical part of the theory.

We can regiment this argument as follows.
(2a) The difficulty of proving a $\Pi^1_n$ sentence
$\forall{X}\exists{Y}\varphi(X, Y)$ is just the difficulty of, given a problem
$X \subseteq \omega$,
computing a solution $Y \subseteq \omega$ such that $\varphi(X, Y)$.
(2b) Computable entailment captures the difficulty of---given a problem
$X \subseteq \omega$---computing a solution $Y \subseteq \omega$ such that
$\varphi(X, Y)$ better than $\rca_0$-provability does.
Therefore, computable entailment captures the difficulty of proving $\Pi^1_n$
sentences better than $\rca_0$-provability.

The content of the intuition (2a) that Shore takes his view to be formalising
is, at least on a naïve reading, somewhat problematic. Firstly, the
mathematician's standard understanding of difficulty of proof does not square
with Shore's account: theorems in number theory are not thought less difficult
to prove just because the solutions are not highly uncomputable. Secondly, while
the notion of degrees of uncomputability is clear, it is arguable whether this
measures difficulty of computation. Since the Turing jump and True Arithmetic
are both uncomputable sets, and thus both intractable problems for finite
computers, even idealised ones with unbounded time and space such as Turing
machines, there is good reason to consider both of these sets equally difficult
to compute, viz., impossible. We shall have more to say on this point in
§\ref{sec:complexity}, but for now let us try to sidestep these issues by
reformulating the argument above.

To do so, we must reject the naïve reading of the two notions that figure in
Shore's equivalence between difficulty of proof and difficulty of computation.
Instead of the mathematician's standard understanding of difficulty of proof,
let us take ``being harder to prove'' to mean ``requiring stronger axioms to
prove'', since this is the standard of difficulty salient to reverse
mathematics.
To resolve the second difficulty, we replace the vague statement ``harder to
compute'' with an explicit reference to the notion of relative computability.
The reformulated argument then runs as follows.
(2a*) The strength of axioms which entail a $\Pi^1_n$ sentence
$\forall{X}\exists{Y}\varphi(X, Y)$ is precisely captured by---given a problem
$X \subseteq \omega$---the degrees of uncomputability of solutions
$Y \subseteq \omega$ such that $\varphi(X, Y)$.
(2b*) Computable entailment captures the degrees of uncomputability---given a
problem $X \subseteq \omega$---of solutions $Y \subseteq \omega$ such that
$\varphi(X, Y)$ better than $\rca_0$-provability does.
Therefore, computable entailment captures the strength of assumptions which
entail $\Pi^1_n$ sentences better than $\rca_0$-provability.

Premise (2b*) is well supported, given the non-absoluteness of Turing
computability across nonstandard models (modulo the usual questions about the
existence of the unique standard natural number structure $\omega$, and the
background theory needed to establish it). We therefore focus on (2a*), which is
distinctly more plausible than (2a), since we already know that the Big Five
subsystems of second order arithmetic are based on computability-theoretic
principles, and that the hierarchy of proof-theoretic strength of these systems
correlates with the Turing degree of solutions whose existence is asserted in
theorems equivalent to these systems.

A contentious point in assessing (2a*) is the status of induction axioms.
\citet{Neeman2011} proves that $\Sigma^1_1$ induction is required to prove that
Jullien's indecomposability theorem implies weak $\Sigma^1_1$ choice: neither
the $\Sigma^0_1$ induction axiom of $\rca_0$, nor even $\Delta^1_1$ induction,
suffices. If we take the classical reverse mathematics framework to privilege
$\rca_0$ as a base theory then this is prima facie problematic. Since
computational reverse mathematics fixes the first order part of the model as the
standard natural numbers $\omega$, it does not have this difficulty.
One problem with this view is that induction axioms are also a form set of
existence principle, namely bounded comprehension schemes for finite sets. The
$\Sigma^0_1$ induction scheme is equivalent over $\rca_*$ to bounded
$\Sigma^0_1$ comprehension: the scheme that for every $n \in \Nat$ the set
$X = \Set{ k < n | \varphi(k) }$ exists, where $\varphi(x)$ is any $\Sigma^0_1$
formula. Stronger induction schemes are likewise equivalent to stronger bounded
comprehension schemes. \citet{SimSmi1986} show that a number of theorems from
algebra are equivalent over $\rca_*$ to $\Sigma^0_1$ induction. This undermines
the idea that classical reverse mathematics should be identified with reverse
mathematics in $\rca_0$, but perhaps this is no bad thing.
Computational reverse mathematics, on the other hand, appears to trivialise
bounded comprehension schemes, since they are all computably entailed. $\Pi^1_n$
theorems, central to all of reverse mathematics, thus become almost the only
subject countenanced (further reverse mathematical equivalences trivialised by
the computable entailment relation are considered in §\ref{sec:justification}).

\subsection{Accessibility of reverse mathematics}
\label{ssec:accessibility_rm}

Another of Shore's motivations in introducing computational reverse mathematics
is expository: making the tools and results of reverse mathematics more
accessible to ordinary mathematicians who do not think, as logicians do, in
terms of formal theories and proof systems. Although it still involves
formalisation, computational reverse mathematics does allow us to sidestep some
formal aspects of reverse mathematics. Instead of proving equivalences to
syntactically defined subsystems of second order arithmetic, we can work
directly with computability-theoretic and combinatorial closure conditions.
Moreover, it allows us to use induction in the standard way, for any property,
not just those definable in the language of arithmetic with a single first order
quantifier.

Making reverse mathematics more accessible to ordinary mathematicians is clearly
a valuable goal, but since computable entailment is not extensionally equivalent
to provability in the standard base theory $\rca_0$, this is not by itself
sufficient to motivate our adoption of computational reverse mathematics over
classical reverse mathematics. In order to provide the relevant motivation, we
could think of this as an instance of naturalistic deferral to mathematical
practice, that is, (3) computable entailment is the right way to measure the
strength of theorems because it better captures the way ordinary mathematicians
work with the objects these theorems concern. In particular, ordinary
mathematicians  ``do not think about (nor care) [\ldots] how much induction is
used in any particular proof'' \citep[p.~381]{Shore2010}. Moreover, definitions
and inferences in informal mathematical practice are not carried out within a
fixed formal framework. Working only with $\omega$-models thus reflects ordinary
mathematical practice as, in practice, mathematicians consider the natural
numbers to be categorical and to satisfy induction for all predicates.

\subsection{Uncountable reverse mathematics}
\label{ssec:uncountable_rm}

Since second order arithmetic does not allow quantification over uncountable
sets, classical reverse mathematics is necessarily limited to treating theorems
concerning objects that are either countable, or can be represented by countable
codes. Computational reverse mathematics offers a way to overcome this barrier
and allow for the development of a reverse mathematics of uncountable
mathematics, by using one of the existing definitions of computation on
uncountable mathematical structures.%
\footnote{%
Others have also proposed ways to extend reverse mathematics to the uncountable,
most notably \citet{Kohlenbach2002, Kohlenbach2005} who has developed an
approach in higher-order arithmetic.
}
This is done by altering the definition of computable entailment to quantify
not over Turing ideals, but over over classes of sets that are closed under a
different notion of relative computability appropriate to the particular
uncountable setting. \citet{Shore2013} takes some initial steps in this
direction by developing a variation of computable entailment that uses
$\alpha$-recursion theory, along with analogues of $\aca$ and $\wkl$ in this
setting, and proving some reverse mathematics-style results for them.

This suggests the following argument.
(4a) A framework for reverse mathematics that allows us to analyse the strength
of theorems throughout ordinary mathematics (i.e. not just countable and
countably-representable mathematics) is superior to one that does not.
(4b) Computational reverse mathematics allows us to carry out such an analysis.
(4c) Classical reverse mathematics does not.
Therefore, computational reverse mathematics is a superior framework in which to
carry out reverse mathematics.

Little work has been done in this setting beyond Shore's initial papers, so its
full promise remains as yet unfulfilled. Nevertheless, the adaptability of
Shore's approach to different settings, with the underlying notion of
computation allowed to vary so as to provide an appropriate measure of the
computational strength of theorems in those settings, suggests a highly
promising route to a reverse mathematics of the uncountable.
That being said, if a computational reverse mathematics of the uncountable turns
out to be a fruitful approach, there is no obvious barrier to developing an
axiomatic counterpart that stands in a similar relation to, for example,
$\alpha$-recursion on uncountable ordinals as $\rca_0$ does to Turing
reducibility.
So even accepting (4a), the status of (4b) and (4c) remains unclear. It is
therefore difficult to evaluate the extent to which its extensability to
uncountable structures weighs in favour of computational reverse mathematics.

Intriguingly, \citet{Shore2010} suggests that computational reverse mathematics
of the uncountable will also provide a testing ground for notions of
computability on uncountable sets, and that (p.~387)
``if a theory of computability for uncountable domains provides a satisfying
analysis of mathematical theorems and constructions in the reverse mathematical
sense based on the approach of [definition \ref{dfn:comp_entail_equiv}], then it
has a strong claim to being a good notion of computation in the uncountable.''
This would result in a virtuous circle of justification. On the one hand, the
success of a notion of computation on uncountable sets in providing a reverse
mathematics of the uncountable would vindicate it as the correct notion of
computation in the uncountable setting. On the other, its status as the correct
notion of computation in the uncountable setting would support its use as the
notion of computation underlying the reverse mathematics of the uncountable.
Alternatively, it may turn out that
``there is no single `right' [notion of computation on uncountable sets] but
that certain ones may be better than others for different branches of
mathematics'' \citep[p.~387]{Shore2010}.
If pursuing computational reverse mathematics could help answer these questions
it would certainly strengthen the case for doing so, but that it might do so
does not provide a prima facie reason for preferring it over classical reverse
mathematics as our general framework for reverse mathematics.

\subsection{Closure conditions and the standard view of reverse mathematics}
\label{ssec:closure}

The standard view in reverse mathematics holds that the significance of
reversals lies in the set existence principles they show to be necessary to
prove ordinary mathematical theorems.%
\footnote{%
The locus classicus of this view is \citet[p.~2]{Simpson2009}.
}
However, the relevant concept of a set existence principle, as used in reverse
mathematics, has not been explicated in any detail. \citet{DeaWal2016} have
argued that such a concept cannot be exhausted by the notion of a comprehension
principle, since this would leave out weak König's lemma and arithmetical
transfinite recursion. In \citep{Eas2016b} I argue that even if one includes
both comprehension principles and separation principles in the concept of a set
existence principle, thereby incorporating all of the Big Five, there are other
mathematically natural set existence principles that are left out, such as weak
weak König's lemma.%
\footnote{%
The restriction of weak König's lemma to trees of positive measure.
}

\citet{Shore2010}'s emphasis on the characterisation of the Big Five in terms of
com\-put\-abil\-ity-theor\-etic closure conditions suggests that we could
understand the concept of a set existence principle in terms of that of a
closure condition on the powerset of the natural numbers. I advance a view along
these lines in \citep{Eas2016b}, where I take closure conditions as being
axiomatized by those $\Pi^1_n$ sentences ($n \geq 2$) that are not equivalent to
any less complex sentence.%
\footnote{%
Here ``equivalent'' is to be understood as meaning provably equivalent in any
suitable base theory $B$ that does not prove the sentence in question, but which
can otherwise be as strong as possible. Sentences of this sort are referred to
in \citep{Eas2016b} as \emph{essentially $\Pi^1_{n\geq2}$ sentences}; see \S 5
of that paper for a detailed account of these notions.
}
This account of set existence principles as closure conditions has a number of
advantages, including its generality, since it not only includes all of the Big
Five and weak weak König's lemma, but also other statements such as choice
principles and Ramsey-like combinatorial statements. It also has an intuitive
appeal, as attested by a number of authors who have identified set existence
principles with closure conditions, albeit without providing a precise account
of which sentences of second order arithmetic axiomatize closure conditions.%
\footnote{%
For instance \citet[p.~8]{Feferman1964}, \citet[p.~451]{Feferman1992},
\citet*[p.~2]{DorDzhHir2015}, and \citet*[p.~864]{ChoSlaYan2014}.
}

Despite these appealing characteristics, the details of the account demonstrate
an oddity, namely the status of induction axioms. Instances of the arithmetical
induction scheme $\Pi^0_\infty\text{-}\mathsf{IND}$ are at most of $\Pi^1_1$
complexity, and thus on this account are not considered to be set existence
principles.%
\footnote{Cf. \citet[pp.~71--2]{Simpson2009}, who argues that ``despite
appearances, the $\Sigma^0_1$ induction axiom of $\rca_0$ can be considered to
be a set existence axiom'', due to its equivalence to the scheme of bounded
$\Sigma^0_1$ comprehension.}
More complex fragments of the full induction scheme, however, such as
$\Sigma^1_1$ induction, will be axiomatized by $\Pi^1_n$ sentences, $n \geq 2$,
and will not be equivalent to less complex sentences except by a base theory
that proves them outright. They should thus, according to the view put forward
in \citep{Eas2016b}, be considered closure conditions, and thus set existence
principles. This is counterintuitive, since instances of induction are typically
taken to concern the structure of the natural numbers, not the structure of its
powerset.

This entanglement between the first and second order parts of the theory arises
because of the existence of non-standard models of arithmetic: the standard
natural numbers $\omega$ satisfy not just the full induction scheme
$\Pi^1_\infty\text{-}\mathsf{IND}$, but the induction scheme formulated in any
language, no matter how expressively powerful. Moreover, all finite sets are
already present in any model satisfying recursive comprehension, since finite
sets are by definition recursive. It is only when we consider non-standard
models that bounded comprehension schemes are required to ensure the existence
of ``finite'' sets of size $k$, where $k$ is a non-standard number.

Shore's view offers us a natural response to this problem. Since in
computational reverse mathematics we take the first order part of the model to
be fixed as the standard natural numbers $\omega$, we can rule out instances of
induction as axiomatizing closure conditions, for they will already be satisfied
by any base theory $B$, as the first order variables of sentences of $B$ will
always range over standard natural numbers. The view that set existence
principles are exactly the closure conditions axiomatized by essentially
$\Pi^1_{n\geq2}$ sentences works smoothly in Shore's framework, without
counterintuitive instances of induction masquerading as closure conditions on
$\PowN$. Computational reverse mathematics thereby seems to offer us a better
account of the concept of a set existence principle than classical reverse
mathematics, and thus a more satisfactory way of vindicating the standard view
that reversals are significant because they demonstrate the set existence
principles necessary to prove theorems of ordinary mathematics.

\section{Preservation of justification under computable entailment}
\label{sec:justification}

Computable entailment collapses many distinctions present under the usual
classical entailment relation, and thus the equivalence classes obtained under
the computable equivalence relation are significantly different from those given
by provable equivalence over $\rca_0$. For instance, the standard natural
numbers satisfy the induction scheme for all predicates in the language of
second order arithmetic. As a result, systems with only restricted induction and
their counterparts with the full induction scheme are computably equivalent. The
presence of full induction is indicated by the absence of the `$0$' subscript in
the system's name: $\rca$ is $\rca_0$ but with full induction, $\wkl$ is
$\wkl_0$ with full induction, and so on. In all cases, the system with full
induction has precisely the same $\omega$-models as its counterpart with
restricted induction, and thus they are computably equivalent.

This presents a problem given the connections between the Big Five and existing
philosophically-motivated programs in the foundations of mathematics. At least
in some cases these subsystems are formalisations in second order arithmetic of
those foundational programs, but it is by no means obvious that the same is
true for other axiom systems which are computably equivalent to them. $\aca_0$
is a predicative system, but the mere fact that $\aca$ is computably equivalent
to it should not compel us to believe that $\aca$ is similarly predicatively
acceptable.

Another way to understand this point is by considering that a key property of
any entailment relation is preserving justification: if we are justified in
accepting the antecedent then we are justified in accepting the consequent. For
computational reverse mathematics to be capable of the foundational analysis
outlined earlier, computable entailment must preserve justification, just as
deductive entailment does. Given any foundational program that we wish to
analyse by proving reverse mathematical results, those results must be justified
on the conception of justification internal to the foundational program itself.
If computable entailment fails to satisfy this requirement then proponents of
such foundational programs will be unmoved by any arguments drawn from
computational reverse mathematics, as they will reject the underlying assumption
necessary to proving the results involved.
In other words, the crux of the issue is not whether computable entailment
preserves justification on some particular account of the epistemology of
mathematics, but whether it respects the justificatory structure of the
foundational programs being analysed.

In §\ref{sec:partial_hilbert} we examined \citet{Simpson1988}'s claim that the
proof-theoretic reducibility of $\wkl_0$ to $\pra$ constitutes a partial
realization of Hilbert's program. There are reasons to question whether
Simpson's interpretation of Hilbert is correct, and plenty of debate to be had
over whether this is in fact a good foundation for mathematics. Nevertheless,
the finitistic reductionism that Simpson proposes is nonetheless a foundational
enterprise worthy of consideration.
One part of such an assessment consists of the use of reverse mathematical
methods to determine the parts of ordinary mathematics that can be developed
within this foundational framework, that is, foundational analysis as studied in
the preceding sections. In order to apply it in this way, our system of reverse
mathematics should therefore be able to analyse Simpson's finitistic
reductionism, and as argued above, that analysis should respect the
justificatory structure of finitistic reductionism.
With this concern in mind, the crucial question is whether or not finitistic
reductionism can be extended from $\wkl_0$ to include all systems $T$ that are
computably equivalent to $\wkl_0$. Only if this is the case can we conclude that
Shore's computational reverse mathematics respects its justificatory structure.

One system that is computably equivalent to $\wkl_0$ is the system $\wkl$. As
mentioned earlier, this system augments $\wkl_0$ with the full induction scheme.
If computable entailment is to preserve justification for the Tait-style
finitist, then $\wkl$ must also be finitistically reducible. But the presence of
the full induction scheme means that, as we shall see below, $\wkl$ proves the
consistency of $\pra$. Therefore, it is not finitistically reducible to $\pra$,
since the canonical formal consistency statement $\Con(\pra)$ is a $\Pi^0_1$
statement that $\pra$ does not (if it is, in fact, consistent) prove. In other
words, it rules out the possibility of a finitistic reduction of the sort
delivered by Sieg for $\wkl_0$, and thus rules out the possibility that $\wkl$
is a finitistically reducible system.

Recall that $\isig{n}$ is the fragment of Peano arithmetic obtained by
restricting the induction scheme to $\Sigma^0_n$ formulae. The following is a
standard result in the literature on first-order arithmetic. A full proof can be
found in \citet[§I.4]{HajPud1993}.

\begin{fact}
    \label{fact:isig_con}
    $\isig{n + 1}$ proves the consistency of $\isig{n}$.
\end{fact}

\begin{cor}
    \
    \begin{enumerate}
        \item $\pra$, $\isig{1}$, $\rca[0]$ and $\wkl[0]$ are equiconsistent.
              \label{eq_con_systems}
        \item $\wkl$ proves the consistency of the systems given in
              (\ref{eq_con_systems}).
        \item $\wkl$ is not $\Pi^0_1$ conservative over the systems given in
              (\ref{eq_con_systems}).
    \end{enumerate}
\end{cor}

\begin{proof}
    $\isig{1}$ is $\Pi^0_2$ conservative over $\pra$ \citep{Parsons1970}; the
    first order part of $\rca_0$ is $\isig{1}$ (that is, they prove the same
    sentences in the language $\mathrm{L}_1$ of first order arithmetic); and
    $\wkl_0$ is $\Pi^1_1$ conservative over $\rca_0$ (this is a result of Leo
    Harrington; a proof appears in \citet[§IX.2]{Simpson2009}). Consequently
    any $\Pi^0_2$ statement provable in $\wkl_0$ (or $\rca_0$ or $\isig{1}$) is
    also provable in $\pra$. Since the canonical consistency statements for
    $\pra$, $\isig{1}$ and $\wkl_0$ are $\Pi^0_1$, any system proving the
    consistency of one of these systems proves the consistency of all the
    others.
    
    By fact \ref{fact:isig_con}, $\isig{2}$ proves the consistency of $\isig{1}$
    and hence the consistency of all the systems listed in
    (\ref{eq_con_systems}). $\wkl$ extends $\isig{2}$ and thus proves all the
    theorems it does. Finally, by the complexity of consistency statements,
    $\wkl$ cannot be $\Pi^0_1$ conservative over any of the systems
    listed in (\ref{eq_con_systems}).
\end{proof}

The methods of infinitary mathematics are justified, according to Simpson's
finitistic reductionism, only to the extent that they are reducible to finitary
ones. This seems to rule out $\wkl$ as a partial realization of Hilbert's
program quite straightforwardly. But if computable entailment preserves
justification, then we are justified in accepting $\wkl$ if and only if we
accept $\wkl_0$, as they are computably equivalent. If this is not the case then
computable equivalence seems to have failed as a way to analyse the mathematical
resources required to derive theorems of ordinary mathematics, since it leads to
underdetermination: we are no longer certain, given some theorem $\varphi$,
whether it is acceptable to the finitistic reductionist if we know only that it
is computably entailed by $\wkl_0$. To resolve this underdetermination one could
prove that $\varphi$ follows from $\wkl_0$ using only resources acceptable to
the finitistic reductionist---but since these resources are simply the axioms of
a finitistically reducible system and the laws of classical logic, this amounts
to simply proving the result in $\wkl_0$, and we are no longer working in
Shore's framework, where all that is necessary to show that one principle
follows from another is to demonstrate that the former is true in every
$\omega$-model of the latter.

This being the case, we have at least one situation in which computational
reverse mathematics is not sufficient to carry out a task in reverse mathematics
of significant philosophical interest and importance. The computable entailment
relation does not always preserve the justificatory structure of foundational
theories, and hence Shore's framework thus cannot be used to conduct the kind of
foundational analysis articulated in the previous chapter---at least for one
important example, namely Simpson's finitistic reductionism.%
\footnote{%
This argument also shows that any modification of reverse mathematics to
strengthen the induction principle of the base theory to include even
$\Sigma^0_2$ induction renders it inappropriate for the foundational analysis of
finitistic reductionism. Friedman's switch from subsystems of $\z_2$ with full
induction in \citep{Friedman1975} to systems with restricted induction in
\citep{Friedman1976} is therefore a crucial one for the foundational analysis of
finitistic reductionism.
}

In fact, we can show more than this: that computable entailment does not
preserve justification for the two other foundational programs we examined
above: Weyl's predicativism and the program of predicative reductionism that
follows from the Feferman--Schütte analysis of predicative provability. Here we
rely on a fact that will be demonstrated in lemma \ref{lem:ce_lower_bound}: if
$\varphi$ is a $\Pi^1_1$ sentence in the language of second order arithmetic,
then $\varphi$ is true if and only if $\models_c \varphi$. We therefore have
that for any code $X$ for a recursive linear order $<_X$, $<_X$ is wellordered
(i.e. $X$ codes a recursive ordinal $\alpha$) if and only if
$\models_c \text{$<_X$ is wellordered}$. Consider the ordinals $\varepsilon_0$
and $\Gamma_0$, which are respectively the proof-theoretic ordinals of $\aca_0$
and $\atr_0$. From an external point of view validating that the recursive codes
for these ordinals really do code wellorderings, we have that
$\mathrm{WO}(\varepsilon_0)$ and $\mathrm{WO}(\Gamma_0)$. But then by the above
fact, $\models_c \mathrm{WO}(\varepsilon_0)$
and $\models_c \mathrm{WO}(\Gamma_0)$.
This means that $\aca_0 \equiv_c \aca_0 + \mathrm{WO}(\varepsilon_0)$
and $\atr_0 \equiv_c \atr_0 + \mathrm{WO}(\Gamma_0)$.

$\aca_0$ does not prove that $\varepsilon_0$ is wellordered, but one might
reasonably consider $\mathrm{WO}(\varepsilon_0)$ to be a predicative principle
nonetheless, since stronger predicative theories, justified on the
Feferman--Schütte analysis of predicativity, do prove it. However, on that
understanding of predicativity, it is not predicatively provable that $\Gamma_0$
is wellordered. There are therefore theories which are computably equivalent to
predicative and predicative reducible ones but which are not themselves either
predicative or predicatively reducible. As we have argued above for the case of
finitistic reductionism, this shows that computational reverse mathematics is
not an appropriate setting in which to carry out the foundational analysis of
predicativism and predicative reductionism.

One might reasonably wonder whether these latter examples are somehow artificial
and do not constitute substantial counterexamples to the preservation of
justification by the computable entailment relation. Since all true $\Pi^1_1$
statements are computably entailed, this question reduces to the question of
whether there are ordinary mathematical theorems that are $\Pi^1_1$ and not
justifiable on the basis of the foundational programs we are considering. The
answer to this is positive, and indeed there are ordinary mathematical theorems
equivalent to statements of the form we have just been considering, namely
$\mathrm{WO}(\alpha)$ for $\alpha < \omega_1^\mathrm{CK}$. \citet{Simpson1988}
showed that Hilbert's basis theorem is equivalent over $\rca_0$ to the
wellordering of $\omega^\omega$. This is the proof-theoretic ordinal of
$\wkl_0$, so $\mathrm{WO}(\omega^\omega)$ is computably entailed by $\wkl_0$ but
not finitistically reducible, since over a weak base theory it implies the
consistency of $\wkl_0$ and thus that of $\pra$. A much stronger example is
Kruskal's theorem, a famous result in graph theory that is equivalent over
$\aca_0$ to the wellordering of the small Veblen ordinal
$\vartheta\Omega^\omega$ \citep[p.~62]{RatWei1993}. This is a recursive ordinal
greater than $\Gamma_0$, and thus $\mathrm{WO}(\vartheta\Omega^\omega)$ is
computably entailed by all predicative and predicatively reducible subsystems of
second order arithmetic, but not predicatively provable.

\section{The complexity of computable entailment}
\label{sec:complexity}

We now turn to a different but related issue with the computable entailment
relation: its completes-theor\-etic complexity. As we know from Church and
Turing's negative answer to the Entscheidungsproblem, the classical provability
relation is uncomputable. Indeed, the set of provable consequences of a theory
like Peano arithmetic is a quintessential example of a recursively enumerable
set that is not recursive. Consequently, while there is no general method for
determining whether or not a sentence $\varphi$ in the language of arithmetic
is provable in $\rca_0$, there is a Turing machine which enumerates the provable
consequences of $\rca_0$, amongst which are the equivalences of classical
reverse mathematics.

Semantic relations such as truth tend to be far more complex than syntactic
relations such as provability, since they are---usually
ineliminably---infinitary in nature. The completeness theorem for classical
first order logic gives us an important counterexample: since
$T \models \varphi \; \Leftrightarrow \; T \vdash \varphi$
for theories $T$ and sentences $\varphi$, we can enumerate the model-theor\-etic
consequences of a theory by enumerating its provable consequences, reducing a
complex semantic relation to a finitary one.
The same does not hold for computable entailment. Not only is it not recursive,
but it is not even arithmetical. As a prelude to demonstrating this, we give a
revised definition of computable entailment, generalised to accommodate
parameters. For the rest of this paper we use the symbol $\Nat$ to refer to the
internal natural numbers of subsystems of second order arithmetic, and reserve
the symbol $\omega$ for the external natural numbers of the metatheory.

\begin{dfn}
    \label{dfn:x_computable_entailment}
    For any set $X \subseteq \omega$, and sentence $\varphi$ in the language
    $\lt$ expanded with a constant symbol for $X$, we say that
    \emph{$\varphi$ is $X$-computably entailed},
    in symbols $\models_c^X \varphi$,
    iff for all Turing ideals $M$ such that $X \in M$,
    $M \models \varphi$.
\end{dfn}

At first glance this may appear less general than the earlier definition, but by
the definition of the satisfaction relation, $(\varphi \models_c^X \psi)$ iff
$\models_c^X (\varphi \rightarrow \psi)$, and the new definition is simpler to work
with in the current context.
Fixing a recursive, bijective Gödel coding of sentences of second order
arithmetic, we represent the computable entailment relation by the set of Gödel
codes for sentences which are computably entailed.
For any $X \subseteq \omega$, let
\begin{equation}
    C(X) = \Set{ \gcode{\varphi} | \models_c^X \varphi }
\end{equation}
where $\varphi$ is an $\lt$-sentence which may contain a constant $\overline{X}$
denoting $X$.
The parameter-free version of $C(X)$ we denote simply $C$.
Observing that the definition of computable entailment quantifies over $\omega$%
-models, we can see that $C$ contains all the sentences of True Arithmetic, the
first order theory of the natural numbers. True Arithmetic is not arithmetically
definable, as this would contradict Tarski's theorem. So computable entailment
cannot be arithmetical either.

A stronger lower bound for the complexity of computable entailment can be found
by noting that arithmetical properties of reals are absolute to all $\omega$-%
models, and thus that all $\Pi^1_1$ sets of natural numbers are $1$-reducible to
$C$. We can thus precisely characterise its complexity as $\Pi^1_1$ complete, by
showing that $C$ can be captured by a $\Pi^1_1$ definition. This theorem is
essentially a classical one due to \citet*[§3.4, pp.~386--7]{GrzMosRyl1958}.
Their result was proved for the second order functional calculus with the
$\omega$-rule, which they refer to as $A_\omega$. We can understand this in the
terminology of the present work as the following result: the set of Gödel
numbers of $\lt$-sentences true in every $\omega$-model of second order
arithmetic $\z_2$ is a $\Pi^1_1$ complete set.
The proof presented below is due to \citet{mummert2012}, who proves it for
$\omega$-models of $\rca_0$ rather than full $\z_2$. Relativizing computable
entailment to a set parameter $X \subseteq \omega$ we have the following.

\begin{thm}
    \label{thm:ce_complexity}
    For any set parameter $X \subseteq \omega$, the $X$-computable entailment
    relation $C(X)$ is $\Pi^1_1(X)$ complete.
\end{thm}

We shall need some standard definitions from computability theory. For more
background the reader should consult a reference work such as \citet{rogers1967}
or \citet{soare1987}.

\begin{dfn}
    For sets $X, Y \subseteq \omega$, \emph{$X$ is many-one reducible to $Y$},
    $X \leq_m Y$, just in case there is a total recursive function $f$
    such that for all $m \in \omega$,
    \begin{equation}
        m \in X \Leftrightarrow f(m) \in Y.
    \end{equation}
    If $f$ is injective then \emph{$X$ is $1$-reducible to $Y$}, $X \leq_1 Y$,
    and if $f$ is a bijection then $X$ and $Y$ are \emph{$1$-equivalent}.
\end{dfn}

\begin{dfn}
    Let $\mathcal{X} \subseteq \mathcal{P}(\omega)$.
    A set $X \subseteq \omega$ is \emph{complete for $\mathcal{X}$} iff
    $X \in \mathcal{X}$ and $Y \leq_1 X$ for every $Y \in \mathcal{X}$.
\end{dfn}

\begin{lem}
    \label{lem:ce_lower_bound}
    For any set parameter $X \subseteq \omega$, every $\Pi^1_1(X)$ set $A$ is
    $1$-reducible to $C(X)$.
\end{lem}

\begin{proof}
    Let $\varphi(m_1, X_1)$ be a $\Pi^1_1$ formula.
    We refer to $(\omega, \mathcal{P}(\omega))$ as the \emph{full model}.
    
    \emph{Claim:} For any $n \in \omega$ and $X \subseteq \omega$,
    $\varphi(n, X)$ is true in the full model iff it is true in all Turing
    ideals containing $X$.
    
    ($\Leftarrow$) The full model is a Turing ideal containing $X$, so if
    $\varphi(n, X)$ is false in the full model then it is false in that ideal.
    
    ($\Rightarrow$) Assume without loss of generality that $\varphi(n, X)
    \equiv \forall{Y}\psi(n, X, Y)$ where $\psi$ is arithmetical.
    Suppose there is a Turing ideal $\mathcal{C}$ containing $X$ such that
    $\mathcal{C} \not\models \varphi(X)$. Then there is some counterexample
    $B \in \mathcal{C}$ such that $\mathcal{C} \not\models \psi(X, B)$. Since
    the interpretation of the first order quantifiers and nonlogical symbols are
    the same in all $\omega$-models, such a $B$ will remain a counterexample in
    the full model.
    
    This completes the proof of the claim.
    
    Given $\varphi(m_1, X_1)$ as above, let
    $A = \Set{ n \in \omega | \varphi(n, X)}$.
    Define the function $f_A : \omega \rightarrow \omega$ as
    $f_A(n) = \gcode{\varphi(\overline{n}, \overline{X})}$. This function is
    recursive and injective, since if $a \neq b$ then
    $\gcode{\varphi(\overline{a}, \overline{X})}
        \neq
    \gcode{\varphi(\overline{b}, \overline{X})}$
    by the properties of the Gödel coding.
    Finally by the claim above and the fact that $\varphi(m_1, X_1)$ is
    $\Pi^1_1$,
    $n \in A
        \leftrightarrow
    \varphi(n, X)
        \leftrightarrow
    \gcode{\varphi(\overline{n}, \overline{X})} = f_A(n) \in C(X)$.
\end{proof}

Having shown that $C$ is $\Pi^1_1$-hard, i.e. that all sets $A \in \Pi^1_1$ are
$1$-reducible to it, we shall show that $C$ is itself $\Pi^1_1$ and thus is
$\Pi^1_1$ complete. In doing so we shall lean on the following definition which
shows how a set can code a countable Turing ideal.
A countable coded $\omega$-model is a set $W$ which codes countable sequence of
sets $\langle (W)_n \mid n \in \mathbb{N} \rangle$ where
$(W)_n = \Set{ i | (i, n) \in W }$. For a full definition of countable
coded $\omega$-models see \citet[§VII.2]{Simpson2009}.

\begin{dfn}
    \label{dfn:ctble_ti}
    Suppose $W \subseteq \Nat$ is a set coding the countable model $M$ and
    $X \subseteq \Nat$.
    \emph{$W$ codes a countable Turing ideal containing $X$} iff
    
    \begin{enumerate}
        \item[(i)]
            For every $m, n$, there exists a $k$
            such that $(W)_k = (W)_m \oplus (W)_n$;
            \label{ctble_ti_join}
        \item[(ii)]
            For every $m$, if $Y \Tleq (W)_m$ then there exists a $k$
            such that $(W)_k = Y$;
            \label{ctble_ti_reduc}
        \item[(iii)]
            There exists some $k$ such that $(W)_k = X$.
            \label{ctble_ti_param}
    \end{enumerate}
\end{dfn}

\begin{lem}
    \label{lem:ti_arithmetical}
    Let $X, W \subseteq \mathbb{N}$.
    The predicate ``$W$ codes a countable Turing ideal containing $X$'' is
    arithmetical.
\end{lem}

\begin{proof}
    Throughout we use the countable coded $\omega$-model $W$ as a parameter. The
    following formula is an analogue of condition (\ref{ctble_ti_join}) of
    definition \ref{dfn:ctble_ti}.
    \begin{equation}
        \begin{aligned}
            \forall{m} \forall{n}
            \exists{k}
            \forall{x} \forall{y} [ \, &
                x \in (W)_m \wedge y \in (W)_n \\
                \leftrightarrow \; &
                2x \in (W)_k \wedge 2y + 1 \in (W)_k
            \, ].
        \end{aligned}
    \end{equation}
    For (\ref{ctble_ti_reduc}), let $\pi(e, n, Y)$ be a universal lightface
    $\Pi^0_1$ formula with the given free variables. The existence of such
    formulae is provable in $\rca[0]$; a definition is provided in
    \citet[definition VII.1.3, p.~244]{Simpson2009}. They play the role of
    universal Turing machines.
    \begin{equation}
        \begin{aligned}
            \forall{m}\forall{e_0}\forall{e_1} [ \, &
                \forall{n} (
                    \pi(e_0, n, (W)_m) \leftrightarrow \neg\pi(e_1, n, (W)_m)
                ) \\
                \rightarrow \; &
                \exists{k}\forall{n} (
                    n \in (W)_k \leftrightarrow \pi(e_0, n, (W)_m)
                )
            \, ].
        \end{aligned}
    \end{equation}
    Finally we add condition (\ref{ctble_ti_param}) that $X$ is an element of
    the Turing ideal coded by $W$,
    \begin{equation}
        \exists{k}\forall{n} (n \in X \leftrightarrow n \in (W)_k).
    \end{equation}
    One can (tediously) verify that these conditions hold of $W$ if and only if
    the $\omega$-model coded by $W$ is a Turing ideal containing $X$.
\end{proof}

\begin{lem}
    \label{lem:ti_countable}
    For any set parameter $X \subseteq \omega$, if an $\lt(X)$-sentence
    $\varphi$ is false in any Turing ideal containing $X$, then it is false in a
    countable Turing ideal containing $X$.
\end{lem}

\begin{proof}
    Let $M$ be a Turing ideal containing $X$, and assume that
    $M \models \neg\varphi$. By the downwards Löwenheim--Skolem theorem, $M$
    has a countable $\omega$-submodel $M' \subseteq_\omega M$ such that
    $X \in M'$. $M'$ is a Turing ideal, as this property is definable by an
    $\lt(X)$ sentence which is true in $M$, and thus in $M'$ by elementarity.
    Finally, $\varphi$ is false in $M'$, again by elementarity.
\end{proof}

\emph{Proof of theorem \ref{thm:ce_complexity}.}
Fix a set parameter $X$. By lemma \ref{lem:ce_lower_bound}, $C(X)$ $1$-reduces
every $\Pi^1_1(X)$ set. It only remains to show that $C(X)$ is itself a
$\Pi^1_1(X)$ set.

Let $C^\dagger(X)$ be the set of Gödel codes of $\lt$-sentences $\varphi$ such
that every countable Turing ideal containing $X$ satisfies $\varphi$. Lemma
\ref{lem:ti_countable} shows that any sentence $\varphi$ of second order
arithmetic is satisfied by every Turing ideal containing $X$ iff it is satisfied
by every countable Turing ideal containing $X$. So
$\gcode{\varphi} \in C(X) \Leftrightarrow \gcode{\varphi} \in C^\dagger(X)$.
Thus by proving that $C^\dagger(X)$ is a $\Pi^1_1(X)$ set, we show that $C(X)$
is also $\Pi^1_1(X)$.

The relation $\gcode{\varphi} \in C^\dagger(X)$ can be defined in second order
arithmetic as:
\begin{equation}
    (\forall{\text{ countable Turing ideals $M$}})
        (X \in M \rightarrow M \models \varphi)
\end{equation}
By lemma \ref{lem:ti_arithmetical}, the predicate ``$W$ codes a countable Turing
ideal $M$'' is arithmetical. $M \models \varphi$ means ``There exists a
satisfaction function $f$ for $M$ such that $f(\gcode{\varphi}) = 1$.'' Although
this is $\Sigma^1_1$, every such $f$ is provably unique, and thus
$M \models \varphi$ is equivalent to a $\Pi^1_1$ formula.
\hfill $\square$

\vspace{10pt}

Computable entailment thus transcends arithmetical truth, being recursively
isomorphic to the $\Pi^1_1$ theory of the natural numbers, and also to
membership in Kleene's $\mathcal{O}$, the set of notations for recursive
ordinals. Nevertheless its complexity is towards the lower end of the logics
considered by \citet{vaananen2001} and \citet{Koellner2010}, being for instance
far less complex than the full second-order consequence relation. But as we
shall soon see, such complexity is incompatible with the requirements of
foundational analysis.

The Entscheidungsproblem was considered by Hilbert and others to be of such
importance because a positive solution would have meant we could obtain, by
finite means, knowledge of the provability or unprovability of all mathematical
statements. The computational intractability of the classical provability
relation constitutes an epistemic difficulty for mathematics. From this
perspective, we should be troubled by an entailment relation such as Shore's
with a far greater degree of uncomputability.

It's well known that truth definitions are not simple: Kripke's fixed-point
construction of a truth predicate over the natural numbers is also
$\Pi^1_1$ complete \citep{kripke1975}.
Provability, at least for classical first-order logic, is comparatively
uncomplicated. If $\rca_0 \models \varphi$ then we can produce a finitary proof
witness by an exhaustive search. We have no such assurance when
$\models_c \varphi$: computable entailment does not satisfy Gödel's completeness
theorem, so we are unable to reduce this complex semantic relation to the more
finitistically acceptable provability relation.

$\omega$-logic does have a completeness theorem of sorts, namely the $\omega$-%
com\-plete\-ness theorem of Henkin and Orey, as mentioned in
§\ref{sec:shore_intro}. By this theorem, restricting to $\omega$-models is
equivalent to closing one's consequence set under the $\omega$-rule. This is
typically formalised in terms of an infinitary proof calculus, where proofs are
well-founded trees which branch infinitely on uses of the $\omega$-rule.
However, this completeness theorem does not induce a reduction in the complexity
of the computable entailment relation: computable entailment is irredeemably
infinitary.
Computable entailment is also impredicative. Shore's definition quantifies over
all Turing ideals, and while theorem \ref{thm:ce_complexity} shows that a
definition quantifying only over countable Turing ideals is in fact equivalent
to Shore's, computable entailment is still $\Pi^1_1$ complete, and thus an
archetypal impredicative relation.

While determining the complexity of the computable entailment relation in a
traditional, computability-theoretic way as we have above is a useful
classification exercise, it comes with some disadvantages. Principally, it does
not make clear what proof-theoretic resources are required in order to prove the
result. This means that it is unclear whether the result is epistemically
accessible to the convinced $\mathcal{F}$-theorist, where $\mathcal{F}$ is a
given foundational approach such as those studied in §\ref{sec:partial_hilbert}
and §\ref{sec:predicativism}. From an external viewpoint we can see that the
computable entailment relation is definable in the language of second order
arithmetic, by quantifying over all countable Turing ideals. We therefore turn
to the resources of reverse mathematics and show that an analogue of theorem
\ref{thm:ce_complexity} can be proved within a predicative subsystem of second
order arithmetic.

To do so we must select the correct base theory, and then formulate the
principle that truth sets for the $X$-computable entailment relations exist with
some care. The first barrier is that $\rca_0$ does not prove that codes for
countable Turing ideals exist. Nor, given a countable coded Turing ideal $M$,
does it prove that the satisfaction function $f$ for $M$ exists. We therefore
work in the stronger theory $\aca_0$. However, $\aca_0$ does not prove that the
full satisfaction function for any countable coded Turing ideal (considered as
an $\omega$-model) exists, since such a satisfaction function is essentially a
truth predicate for the first-order language of arithmetic, and thus not
arithmetically definable. We therefore formulate the definition of the truth set
for $X$-computable entailment $C(X)$ in a slightly modified form, using not full
satisfaction functions but valuation functions for single sentences.%
\footnote{%
This definition is extensionally equivalent to the previous one using full
satisfaction functions, as can be seen from the viewpoint of proof-theoretically
stronger but still predicatively reducible theories such as $\aca_0^+$, which
prove that satisfaction functions for countable coded $\omega$-models exist.
\citet{Dor2012} explains many of the subtleties involved.
}
For details of the notion of a valuation function see definition VII.2.1 of
\citet{Simpson2009}.

\begin{dfn}
    \label{dfn:x_computable_entailment_aca0}
    The following definition is made in $\aca_0$.
    Let $X \subseteq \Nat$ be any set and let $\gcode{\varphi}$ be a Gödel code
    for a sentence $\varphi$ in the language of second order arithmetic $\lt$
    extended with a constant symbol for $X$.
    We say that $\varphi$ is \emph{$X$-computably entailed},
    $\gcode{\varphi} \in C(X)$,
    if for every code $W$ for a countable Turing ideal $M$ such that $X \in M$,
    and for every valuation function
    $f : \mathrm{Sub}_M(\varphi) \rightarrow \Set{ 0, 1 }$,
    we have that $f(\gcode{\varphi}) = 1$.
\end{dfn}

\begin{lem}
    \label{lem:ce_complexity_aca0}
    Suppose $\varphi(m, X)$ is a $\Pi^1_1$ formula with exactly the displayed
    free variables.
    Then the following is provable in $\aca_0$.
    For all $X \subseteq \Nat$, if $C(X)$ exists, then
    $Y = \Set{ m | \varphi(m, X) }$ exists and $Y \Tleq C(X)$.
\end{lem}

\begin{proof}
    Let $\varphi(m, X)$ be as above.
    By the Kleene normal form theorem for $\Pi^1_1$ formulas, there is a
    $\Sigma^0_1$ formula $\sigma(m, f, X)$ with exactly the displayed free
    variables such that $\aca_0$ proves
    \begin{equation*}
        \forall{m}\forall{X}(
            \varphi(m, X) \leftrightarrow \forall{f}\sigma(m, f, X)
        ).
    \end{equation*}
    Given $m \in \Nat$ and $X \subseteq \Nat$, we reason in $\aca_0$ and show
    that
    \begin{equation*}
        \forall{f}\sigma(m, f, X)
        \leftrightarrow
        \gcode{\forall{f}\sigma(m, f, X)} \in C(X).
    \end{equation*}
    
    ($\Rightarrow$)
    Suppose $\gcode{\forall{f}\sigma(m, f, X)} \not\in C(X)$.
    Then there exists a code $W_1$ for a countable Turing ideal $M_1$ containing
    $X$, and a valuation function
    $g_1 : \mathrm{Sub}_{M_1}(\psi_1) \rightarrow \Set{0, 1}$
    (where $\psi_1 \equiv \forall{f}\sigma(m, f, X)$),
    such that $g_1(\psi_1) = 0$.
    By the definition of the valuation function, there exists $k \in \Nat$ which
    is the index of a function $f_1 = (W)_k$,
    such that $g_1(\gcode{\sigma(m, f_1, X)}) = 0$.
    $f_1$ exists by recursive comprehension in the parameter $W$, and since
    arithmetical formulas are absolute between the ambient model and any
    countable coded $\omega$-model, we have that $\neg\sigma(m, f_1, X)$, and so
    $\neg\forall{f}\sigma(m, f, X)$.
    
    ($\Leftarrow$)
    Suppose there exists $f_2$ such that $\neg\sigma(m, f_2, X)$.
    By arithmetical comprehension there exists a code $W_2$ for a countable
    Turing ideal $M_2$ such that $X, f_2 \in M_2$ (for example, just take the
    code for the ideal consisting of the sets recursive in $X \oplus f_2$).
    Arithmetical comprehension also proves the existence of a valuation function
    $g_2 : \mathrm{Sub}_{M_2}(\psi_2) \rightarrow \Set{0, 1}$, where $\psi_2
    \equiv \neg\sigma(m, f_2, X) \wedge \forall{f}\sigma(m, f, X)$.
    By absoluteness, $g_2(\gcode{\neg\sigma(m, f_2, X)}) = 1$,
    so $g_2(\gcode{\forall{f}\sigma(m, f, X)}) = 0$,
    and hence $\gcode{\forall{f}\sigma(m, f, X)} \not\in C(X)$.
    
    Given a set $X \subseteq \Nat$, assume that $C(X)$ exists. Let $Y$ be the
    set of all $m$ such that $\gcode{\forall{f}\sigma(m, f, X)} \in C(X)$.
    $Y$ is recursive in $C(X)$, and thus exists by recursive comprehension.
    So by the equivalence just proved,
    $\forall{m}(m \in Y \leftrightarrow \varphi(m, X))$.
\end{proof}

While this result demonstrates that the complexity of the computable entailment
relationship is in some sense accessible to the predicativist, if not the
finitistic reductionist, it is not clear what the philosophical moral should be.
A natural interpretation might be that a $\Pi^1_1$ complete entailment relation
such as computable entailment is much more uncomputable than the relation of
provability in classical logic, and that this degree of uncomputability can be
comprehended from within a predicativist formal theory. This conclusion goes
hand in hand with a more general view that the Turing degrees track a hierarchy
of relative difficulty of problem solving: problems with higher Turing degree
are harder to solve than those with lower Turing degree.

This is related to the naïve reading of the intuition discussed in
§\ref{ssec:complexity_difficulty} that ``being harder to prove'' means ``harder
to compute'', and is an unsatisfying interpretation for similar reasons. The
first concerns the use of the computable entailment relation. We are not seeking
a general method that for any $\varphi, \psi$ in the language of second order
arithmetic tells us whether or not $\varphi \models_c \psi$. Rather, given
specific statements of mathematical interest, we try to prove (or disprove) that
one computably entails the other. The proofs involved here are typical
mathematical proofs, carried out in the usual way, not infinitary inferences:
they quantify over Turing ideals, but they do not require that we are able, as
mathematicians, to solve the halting problem or determine membership in Kleene's
$\mathcal{O}$.
Secondly, it is unclear why---given that they are both uncomputable---that the
computable entailment relation is epistemically any more intractable than the
standard first order provability relation of classical logic. Absent the ability
to carry out supertasks, we cannot solve the halting problem, so even in
principle, determining for arbitrary $\varphi, \psi$ whether or not
$\rca_0 \vdash \varphi \rightarrow \psi$
seems as out of reach as determining whether
$\varphi \models_c \psi$.

Given this, a more plausible reconstruction of mathematical practice when we
prove computable entailments or failures of computable entailment is that we
work in a way that can be formalised in a standard deductive calculus, but that
in doing so we assume that quantifying over all Turing ideals (or equivalently,
over all countable Turing ideals) is well-defined. One way of ensuring this
well-definedness is to work in a background theory that proves that the
extension of the computable entailment relation exists.
In this context, knowing the precise Turing degree of the computable entailment
relation takes on greater significance, since it allows us to determine what
axioms are both necessary and sufficient to prove its well-definedness.

\begin{cor}
    \label{thm:equiv_piicao_compent}
    The following are equivalent over $\aca_0$.
    \begin{enumerate}
        \item $\Pi^1_1$ comprehension.
        \label{thm_equiv_piicao_compent_a}
        \item For every $X \subseteq \Nat$, the truth set $C(X)$
        of the $X$-computable entailment relation exists.
        \label{thm_equiv_piicao_compent_b}
    \end{enumerate}
\end{cor}

\begin{proof}
    The definition of $C(X)$ is $\Pi^1_1$, as can be seen from definition
    \ref{dfn:x_computable_entailment_aca0}, so $\Pi^1_1$ comprehension proves
    that for any $X \subseteq \Nat$, $C(X)$ exists.
    
    For the reversal we work in $\aca_0$.
    Let $\varphi(m)$ be a $\Pi^1_1$ formula.
    Using pairing, join all $n$ free set variables in $\varphi(m)$, to produce
    an equivalent formula $\varphi'(m, Y)$ with a single free set variable $Y$,
    such that $\aca_0$ proves
    \begin{equation*}
        \forall{X_1}\dotsc\forall{X_n}\forall{Y}\forall{m}(
            \varphi(m) \leftrightarrow \varphi'(m, Y)
        ).
    \end{equation*}
    Given $Y$, assume that $C(X)$ exists.
    Lemma \ref{lem:ce_complexity_aca0} applied to $\varphi'(m, Y)$, together
    with the above equivalence, implies the existence of a set $Z$ such that
    $\forall{m}(m \in Z \leftrightarrow \varphi(m))$.
    This proves $\Pi^1_1$ comprehension.
\end{proof}

$\pica_0$ is the strongest of the subsystems of second order arithmetic usually
studied in reverse mathematics. Computational reverse mathematics therefore
draws on resources which are unavailable in the four members of the Big Five
that are proof-theoretically weaker than $\pica_0$.
Moreover, since \ref{thm:equiv_piicao_compent} is provable within a predicative
system, the predicativist is in a position to calibrate the strength of the
commitment involved in accepting computable entailment. Doing so, she will see
that not only is it stronger than predicative systems like $\aca_0$, but also
predicatively reducible ones like $\atr_0$. So not only does the existence of
the truth set for the computable entailment relation exceed the strength of the
predicativist and the predicative reductionist's theoretical resources, but they
are in a position to see that it does. Since they reject impredicative
mathematics, and thus reject $\Pi^1_1$ comprehension, they must therefore reject
the equivalent statement that the truth set for computable entailment exists.

For foundational analysis to be a useful and worthwhile endeavour within the
philosophy of mathematics, the fruits of its analysis must be epistemically
available to disputants. Recall our example of Sarah the predicativist from
§\ref{sec:foundational_analysis}. Since she accepts $\aca_0$, she believes that
the equivalence between $\Pi^1_1$ comprehension and the statement ($\star$),
``Every countable abelian group can be expressed as a direct sum of a divisible
group and a reduced group'' is true, since it is provable in a system contained
in $\aca_0$ (namely $\rca_0$). How she responds to Rebecca's challenge that
Sarah's predicativism is misguided, since it does not allow her to prove the
ordinary mathematical theorem ($\star$), will depend on the details of her views
about the foundations of mathematics, but she cannot dismiss the equivalence as
question-begging. On the other hand, suppose Rebecca were instead to present
Sarah with the following argument: $\pica_0$ and ($\star$) are computably
equivalent, that is to say they are true in exactly the same Turing ideals.
Sarah should therefore accept $\pica_0$, since ($\star$) is an ordinary
mathematical theorem that any reasonable foundational system should prove. In
this case Sarah can resist the conclusion by refusing to accept the antecedent:
computable equivalence is not a well-defined notion, since it presupposes
theoretical resources which predicativism denies. Any argument presupposing that
computable equivalence is a well-defined notion therefore begs the question
against her position.

We argued in §\ref{sec:foundational_analysis} that philosophical arguments that
attempts to invoke reverse mathematical results in foundational analysis should,
if they are to have any force, appeal only to principles that targets of these
argument already accept. In other words, its presuppositions must not exceed
their theoretical commitments. But the argument above shows that the theoretical
commitments which accompany the use of computable entailment outstrip those
acceptable to partisans of most of the foundational programs associated with
subsystems of second order arithmetic. Computational reverse mathematics does
not allow one, in general, to persuasively demonstrate the mathematical limits
of these foundational programs to those who accept them.

\section*{Funding}

This work was supported by the Arts and Humanities Research Council [doctoral
studentship, 2011–14]; and the Leverhulme Trust [International Network grant
``Set Theoretic Pluralism''].

\section*{Acknowledgements}

This article is drawn from my doctoral research at the University of Bristol.
I am indebted to Carl Mummert for his proof of theorem \ref{thm:ce_complexity}
and for his patient explanations of the result and related matters.
I would also like to thank Kentaro Fujimoto, who helped me determine the correct
base theory for lemma \ref{lem:ce_complexity_aca0}, and François G. Dorais, for
explaining the subtleties of satisfaction relations in reverse mathematics.
I am grateful to the journal editor and two anonymous referees whose comments
helped me to substantially improve this paper.
Finally I would like to thank
Marianna Antonutti Marfori,
Walter Dean,
Leon Horsten,
Richard Kaye,
Øystein Linnebo,
Toby Meadows,
Richard Pettigrew,
Sam Sanders,
Sean Walsh,
and Philip Welch,
for fruitful discussions and their many helpful comments,
as well as the audiences of talks in Lund, Bristol, Munich, Birmingham,
and Vienna, where I presented earlier versions of this material.


\end{document}